\documentclass{article}
\usepackage{amssymb,pstricks}
\usepackage{latexsym}
\usepackage{amsmath}
\usepackage{amsthm}

\topskip  \theoremstyle{remark}
\theoremstyle{plain}
\newtheorem{theorem}{Theorem}[section]
\newtheorem{lemma}[theorem]{Lemma}

\begin{document}
\title{On ascent sequences avoiding 021 and a pattern of length four}
\author{Toufik Mansour\\
\small Department of Mathematics, University of Haifa, 3498838 Haifa, Israel\\[-0.8ex]
\small\texttt{tmansour@univ.haifa.ac.il}\\[1.8ex]
Mark Shattuck\\
\small Department of Mathematics, University of Tennessee, 37996 Knoxville, TN\\[-0.8ex]
\small\texttt{mark.shattuck2@gmail.com}\\[1.8ex]
}
\date{\small }
\maketitle

\begin{abstract}
Ascent sequences of length $n$ avoiding the pattern $021$ are enumerated by the $n$-th Catalan number $C_n=\frac{1}{n+1}\binom{2n}{n}$.  In this paper, we extend this result and enumerate ascent sequences avoiding $\{021,\tau\}$, where $\tau$ is a pattern of length four.  We in turn identify all of the corresponding Wilf-equivalence classes and find generating function formulas corresponding to each class.  In a couple of cases, we make use of an auxiliary statistic and the kernel method to ascertain the generating function. In several cases, our work of enumeration is shortened by establishing the equivalence of $\{021,\tau\}$- and $\{021,\tau'\}$-avoiders of a given length through an explicit bijection.  As a consequence of our results, one obtains new combinatorial interpretations in terms of ascent sequences for several of the entries in the OEIS.
\end{abstract}



\section{Introduction}

Given a sequence $w=w_1\cdots w_n$, let $\text{asc}(w)$ denote the number of \emph{ascents} of $w$, i.e., the number of indices $1 \leq i \leq n-1$ such that $w_i<w_{i+1}$.  An \emph{ascent sequence} $a_1\cdots a_n$ is a word on the alphabet of non-negative integers such that $a_1=0$ and $a_i \leq \text{asc}(a_1\cdots a_{i-1})+1$ for $1<i\leq n$.  For example, the sequence $0101243503$ is an ascent sequence, whereas $0100125$ is not since 5 exceeds $\text{asc}(010012)+1=4$.   Ascent sequences were introduced in \cite{BCD}, where they were related to the (2+2)-free posets and the generating function was determined,  and have subsequently been studied with connections being made to several other discrete structures.  See, e.g., \cite{CDK,DP,KR} as well as \cite[Section 3.2.2]{K} for further information.

By a \emph{pattern}, we mean a sequence of non-negative integers, where repetitions are allowed, that contains all of the letters in $\{0,1,\ldots,\ell\}$ for some $\ell \geq0$.  Let $\pi=\pi_1\pi_2\cdots \pi_n$ be an arbitrary sequence of non-negative integers and let $\tau=\tau_1\tau_2\cdots\tau_m$ be a pattern.  Then it is said that $\pi$ \emph{contains} $\tau$ (in the classical sense) if there exists a subsequence of $\pi$ that is order isomorphic to $\tau$.  More precisely, there exist indices $i_1,\ldots,i_m$ where $1\leq i_1<\cdots<i_m\leq n$ such that $\pi_{i_j}\,x\,\pi_{i_k}$ if and only if $\tau_j\,x\,\tau_k$ for all $j,k \in [m]=\{1,\ldots,m\}$ and $x \in \{<,>,=\}$.  If no such subsequence occurs, then $\pi$ is said to \emph{avoid} $\tau$.  For example, the ascent sequence $01013102412$ has three occurrences of the pattern $100$, namely, the subsequences $100$, $311$ and $322$, but is seen to avoid the pattern $1230$.  Let $A_n$ be the set of all ascent sequences of length $n$ and let $A_n(\alpha_1,\ldots,\alpha_r)$ denote the subset of $A_n$ whose members avoid each of the patterns $\alpha_1,\ldots,\alpha_r$.

The pattern avoidance problem on ascent sequences was first considered by Duncan and Steingr\'{\i}msson \cite{DS}, who established several results involving a single pattern of length up to four; see also \cite{Chen,MS1,Y} for further related results.  The problem of avoiding two patterns of length three was later addressed by Baxter and Pudwell  \cite{BP}, which was extended to triples of length three patterns in \cite{CM0,CM}.  Further results in this direction include those of Callan et al.~\cite{CMS}, who found all pattern pairs $(\alpha,\beta)$ of length four such that $|A_n(\alpha,\beta)|=C_n$ for $n\geq1$, and of Mansour and Shattuck \cite{MS2}, who found all pairs of length three and four such that $|A_n(\alpha,\beta)|=F_{2n-1}$, where $F_m$ denotes the $m$-th Fibonacci number.

Let $B_n=A_n(021)$. In \cite{DS}, it was shown that $|B_n|=C_n$ for all $n\geq0$.  This result was extended by Chen et al.~\cite{Chen}, who showed via a bijection that the bivariate distribution of the ascents and right-to-left minimum statistics on $B_n$ is the same as that of the corresponding distribution on the set of 132-avoiding permutations of length $n$ for all $n$, which answers an earlier question raised in \cite{DS}.  Note that an ascent sequence belongs to $B_n$ if and only if the subsequence consisting of its positive letters is nondecreasing. In this paper, we consider the pattern-avoidance problem on $B_n$ and find enumerative formulas for several restricted subsets of $B_n$.

Given a pattern $\tau$, let $B_n(\tau)=A_n(021,\tau)$ for $n \geq 1$, with $B_0(\tau)$ consisting of the single empty ascent sequence of length zero (which will often be denoted by $\varepsilon$).  Let $b_n(\tau)=|B_n(\tau)|$ for $n\geq0$ and $B(\tau)=\cup_{n\geq0}B_n(\tau)$.  Define the generating function $f_\tau=f_\tau(x)$ by $f_\tau=\sum_{n\geq0}b_n(\tau)x^n$; that is, $f_\tau$ enumerates all members of $B(\tau)$ according to length.  Here, we find $f_\tau$ when $|\tau|=4$ and thereby determine all of the Wilf-equivalence classes corresponding to a single pattern of length four on $B_n$.  See \cite{MY}, where an analogous problem concerning the enumeration of inversion sequences avoiding 021 and another pattern is investigated.

The organization of this paper is as follows. In the next section, we find $f_\tau$ in all cases where $\tau$ is of length four and whose Wilf-equivalence class is nontrivial.  In several instances, we define bijections between $B_n(\tau)$ and $B_n(\tau')$ to complete the proof of the result, once the generating function is ascertained in one of the cases.  In the final section, we find $f_\tau$ for the remaining $\tau$ for which the corresponding Wilf class is trivial.  Taken together with the results from the second section, one then has an explicit formula for $f_\tau$ in all cases where $\tau$ is of length four.  The generating function corresponding to each nontrivial Wilf-equivalence class is given below in Table \ref{tab1}, along with the first several terms in its Taylor series expansion.

\begin{table}[htp]
{\fontsize{7}{7}\selectfont
\begin{tabular}{|l|l|}\hline
$\{0010,0100,0110,0123\}$& $\frac{(1-x)(1-4x+4x^2+x^3)}{(1-2x)^3}$\\
&$1+x+2x^2+5x^3+13x^4+34x^5+88x^6+224x^7+560x^8+1376x^9+3328x^{10}+7936x^{11}+\cdots$\\\hline
$\{0011,0112\}$&$\frac{1-4x+5x^2-x^3}{(1-x)(1-2x)^2}$\\
&$1+x+2x^2+5x^3+13x^4+33x^5+81x^6+193x^7+449x^8+1025x^9+2305x^{10}+5121x^{11}+\cdots$\\\hline
$\{0012,0102\}$& $\frac{1-4x+6x^2-3x^3+x^4}{(1-x)^3(1-2x)}$\\
&$1+x+2x^2+5x^3+13x^4+32x^5+74x^6+163x^7+347x^8+722x^9+1480x^{10}+3005x^{11}+\cdots$\\\hline
$\{0101,0120,0122\}$& $\frac{1-2x}{1-3x+x^2}$\\
&$1+x+2x^2+5x^3+13x^4+34x^5+89x^6+233x^7+610x^8+1597x^9+4181x^{10}+10946x^{11}+\cdots$\\\hline
$\{1000,1100\}$& $\frac{(1-x)(1-8x+24x^2-31x^3+13x^4+2x^5+2x^6)}{(1-2x)^5}$\\
&$1+x+2x^2+5x^3+14x^4+41x^5+122x^6+362x^7+1060x^8+3048x^9+8592x^{10}+23744x^{11}+\cdots$\\\hline
$\{1001,1011,1101\}$& $\frac{1-3x+3x^2-\sqrt{(1-3x+x^2)^2-4x^3(1-x)}}{2x^2}$\\
&$1+x+2x^2+5x^3+14x^4+41x^5+123x^6+376x^7+1168x^8+3678x^9+11716x^{10}+37688x^{11}+\cdots$\\\hline
$\{1010,1120\}$&$\frac{(1-x)(1-5x+7x^2-x^3)}{(1-2x)^2(1-3x+x^2)}$\\
&$1+x+2x^2+5x^3+14x^4+41x^5+121x^6+354x^7+1021x^8+2901x^9+8130x^{10}+22513x^{11}+\cdots$\\\hline
$\{1020,1022\}$& $\frac{1-9x+32x^2-56x^3+49x^4-19x^5+x^6}{(1-x)(1-2x)^3(1-3x+x^2)}$\\
&$1+x+2x^2+5x^3+14x^4+41x^5+120x^6+344x^7+961x^8+2620x^9+6996x^{10}+18369x^{11}+\cdots$\\\hline
$\{1200,1220,1230\}$& $\frac{1-7x+17x^2-16x^3+5x^4-x^5}{(1-2x)(1-3x+x^2)^2}$\\
&$1+x+2x^2+5x^3+14x^4+41x^5+122x^6+364x^7+1082x^8+3195x^9+9362x^{10}+27219x^{11}+\cdots$\\\hline
\end{tabular}}
\caption{The nontrivial Wilf-equivalence classes for patterns of length four on $B_n$, along with the corresponding generating functions.}\label{tab1}
\end{table}

\section{Nontrivial Wilf classes for patterns of length four}

In this section, we establish the generating function formulas given in Table \ref{tab1} above.  Our work is shortened somewhat in this regard by making the following preliminary observations.  First, the generating function given for the fourth entry corresponding to the Wilf class containing the patterns $0101$, $0120$ and $0122$ is already known and occurs as part of a larger result. See \cite[Proposition 3.3]{MS2}, which gives each pair $(\alpha,\beta)$ where $|\alpha|=3$ and $\beta|=4$ such that $|A_n(\alpha,\beta)|=F_{2n-1}$ for all $n\geq1$.

Furthermore, the third entry in Table \ref{tab1} corresponding to the patterns $0012$ and $0102$ may be reasoned as follows.  By \cite[Lemma 2.4]{DS}, since both $0012$ and $0102$ are contained in the larger pattern $01012$, members of $A_n$ avoiding either pattern actually correspond to set partitions, represented sequentially on the non-negative integers as restricted growth functions.  Let $P_n$ denote the set of all partitions of $[n]$ and $P_n(\alpha_1,\ldots,\alpha_r)$ the subset of $P_n$ whose members when represented sequentially avoid each of the patterns $\alpha_1,\ldots,\alpha_r$. Then we have that the members of $B_n(0012)$ and $B_n(0102)$ are synonymous with those in $P_n(0012,0121)$ and $P_n(0102,0121)$, respectively, as avoiding $021$ is the same as avoiding $0121$ for set partitions due to the restricted growth condition.  Combining Theorem 4.9, Example 4.16 and Proposition 4.23 from \cite{JMS}, the latter two pattern pairs are seen to be equivalent on $P_n$ for all $n$ with the generating function formula as stated.

We now turn our attention to the remaining Wilf classes in Table \ref{tab1} and establish the first entry. By a \emph{staircase} word, we will mean throughout a sequence $w=w_1w_2\cdots$ such that $w_{i+1}=w_i$ or $w_i+1$ for all $i$.  Given a nonnegative integer $m$, let $\mathcal{T}_m$ denote the set of all (nonempty) staircase words of finite length having first letter $m$.

\begin{theorem}\label{0010th}
If $\tau \in \{0010,0100,0110,0123\}$, then
\begin{equation}\label{0010the1}
f_\tau(x)=\frac{(1-x)(1-4x+4x^2+x^3)}{(1-2x)^3}.
\end{equation}
\end{theorem}
\begin{proof}
We first compute $f_{0010}$.  To do so, note that $\pi \in B(0010)$ is either nondecreasing or of one of the following two forms:
\begin{align*}
&\bullet\,01^{a_1}\cdots j^{a_j}0^r\sigma, \quad j \geq 1,\\
&\bullet\,01^{a_1}\cdots j^{a_j}0^rj^t(j+1)^{a_{j+1}}\cdots\ell^{a_\ell}\rho, \quad 1 \leq j \leq \ell,
\end{align*}
where all exponents indicating run lengths of the various letters are positive, $\sigma$ is empty or belongs to $\mathcal{T}_j$ or $\mathcal{T}_{j+1}$ and $\rho \in \mathcal{T}_{\ell+2}$.  Combining the prior cases above implies
\begin{align*}
f_{0010}&=\frac{1-x}{1-2x}+\sum_{j\geq1}\frac{x^{j+2}}{(1-x)^{j+1}}\left(1+\frac{2x}{1-2x}+\frac{x^2}{(1-2x)^2}\right)\\
&=\frac{1-x}{1-2x}+\frac{x^2(1-x)}{(1-2x)^2}\sum_{j\geq1}\left(\frac{x}{1-x}\right)^j=\frac{1-x}{1-2x}+\frac{x^3(1-x)}{(1-2x)^3},
\end{align*}
which yields the stated formula for $f_\tau$ in the case $\tau=0010$.

To complete the proof of \eqref{0010the1}, we define a bijection between $B_n(0010)$ and $B_n(\tau)$ for the three remaining patterns $\tau$.  To define a bijection with $B_n(0100)$, consider replacing the section $01^{a_1}\cdots j^{a_j}0^r$ with $0^r1^{a_1}\cdots j^{a_j}0$ in the two forms stated above for $\pi \in B_n(0010)$ containing a descent. For the pattern $\tau=0110$, we proceed as follows in defining a mapping $g$ between $B_n(0010)$ and $B_n(0110)$.  Let $\pi \in B_n(0010)$.  If $\pi=0^{a_0}1^{a_1}\cdots j^{a_j}$ for some $j \geq0$, then let $g(\pi)=0^{a_0}10^{a_1-1}\cdots j0^{a_{j}-1}$.  So assume $\pi$ contains a second run of 0 and is of one of the two forms stated above.  If at least one $j$ occurs to the right of the string $0^r$ in either of the two forms, then replace the initial section $01^{a_1}\cdots j^{a_j}0^r$ with $0^r10^{a_1-1}\cdots(j-1)0^{a_{j-1}-1}j0^{a_j}$ in either case.  Otherwise, if no $j$ occurs to the right of the string $0^r$, i.e., if $\pi=01^{a_1}\cdots j^{a_j}0^r\alpha$, where $\alpha$ is either empty or belongs to $\mathcal{T}_{j+1}$, then let $g(\pi)=0^r10^{a_1-1}\cdots(j-1)0^{a_{j-1}-1}j^{a_j+1}\alpha$.  Note that $g$ may be reversed by considering whether or not a member of $B_n(0110)$ contains a repeated positive letter, and if so, whether or not a 0 occurs between the first and second occurrence of the smallest such letter. Hence, we have that $g$ provides the desired bijection between $B_n(0010)$ and $B_n(0110)$.

We now turn to $\tau=0123$.  First, given a staircase word $w=w_1\cdots w_m$ and a fixed $r\geq1$, let $c_r(w)$ denote the word $c=c_1\cdots c_m$ on $\{0,r\}$ where $c_1=r$ and $c_j$ for $2 \leq j \leq m$ is equal to $r$ or $0$ depending on whether or not $w_j=w_{j-1}$ or $w_j=w_{j-1}+1$. To define a bijection $h$ between $B_n(0010)$ and $B_n(0123)$, we first pair certain members of the respective sets accordingly as follows:
\begin{align*}
&\bullet\, 0^{a_0}1^{a_1}\cdots k^{a_k} \longleftrightarrow 0^{a_0}10^{a_1-1}\cdots10^{a_{k}-1}, \quad k \geq0,\\
&\bullet\, 01^{a_1}\cdots j^{a_j}0^rj^s\rho \longleftrightarrow 0^{r+1}1^{a_1}0^{a_2}\cdots1^{a_{j-1}}0^{a_j}1^sc_2(\rho), \quad j \text{ even,}\\
&\bullet\, 01^{a_1}\cdots j^{a_j}0^rj^s\rho \longleftrightarrow 01^{a_1}0^{a_2}\cdots1^{a_{j-2}}0^{a_{j-1}}1^{a_j}0^r1^sc_2(\rho), \quad j \text{ odd,}\\
&\bullet\, 01^{a_1}\cdots j^{a_j}0^rj^s \longleftrightarrow 0^{s+1}1^rc_2(1^{a_1}\cdots j^{a_j}),
\end{align*}
where $j,r\geq 1$, $s \geq0$, $a_i\geq1$ for each $i$ and $\rho \in \mathcal{T}_{j+1}$.  Note that members of $B_n(0010)$ in the last three forms must start with a single zero, as there is a second run of $0$ (whose length is unrestricted).

Now suppose $\pi \in B_n(0010)$ is expressible as
$$\pi=01^{a_1}\cdots j^{a_j}0^rj^t(j+1)^{a_{j+1}}\cdots\ell^{a_\ell}\alpha$$
for some $1 \leq j \leq \ell$ and $r,t \geq 1$, where $\alpha \in \mathcal{T}_{\ell+2}$.  We define $h$ for such $\pi$ by considering cases on the parity of $\ell$ and $j$ as follows:
\begin{align*}
&\bullet\,\ell,j \text{ even:}\,\,\, \pi \leftrightarrow 0^{t+1}1^{a_1}0^{a_2}\cdots1^{a_{\ell-1}}0^{a_\ell}1^rc_{\frac{j}{2}+2}(\alpha), \\
&\bullet\,\ell \text{ even,}\,\,j \text{ odd:}\,\,\, \pi \leftrightarrow 01^{a_1}0^{a_2}\cdots1^{a_{\ell-1}}0^{a_\ell}1^r0^tc_{\frac{j+5}{2}}(\alpha), \\
&\bullet\,\ell \text{ odd,}\,\,j \text{ even:}\,\,\, \pi \leftrightarrow 0^{t+1}1^{a_1}0^{a_2}\cdots1^{a_{\ell-2}}0^{a_{\ell-1}}1^{a_\ell}0^rc_{\frac{j}{2}+2}(\alpha),\\
&\bullet\,\ell,j \text{ odd:}\,\,\, \pi \leftrightarrow 01^{a_1}0^{a_2}\cdots1^{a_{\ell-2}}0^{a_{\ell-1}}1^{a_\ell}0^r1^tc_{\frac{j+5}{2}}(\alpha).
\end{align*}
Note that if $\ell$ is even in the preceding cases, then there are $\frac{\ell}{2}+1$ runs of 1 (and hence the same number of ascents) prior to the terminal section $c_{\lfloor(j+5)/2\rfloor}(\alpha)$ of $h(\pi)$, with $3 \leq \lfloor\frac{j+5}{2}\rfloor \leq \frac{\ell}{2}+2$. Thus, $h$ is well-defined in these cases.  Likewise, if $\ell$ is odd, then there are $\frac{\ell+1}{2}$ runs of $1$ in $h(\pi)$ if $j$ is even, with $3 \leq \frac{j}{2}+2 \leq \frac{\ell+3}{2}$, and there are $\frac{\ell+3}{2}$ runs of $1$ if $j$ is odd, with $3 \leq \frac{j+5}{2} \leq \frac{\ell+5}{2}$.  One may verify that $h$ is reversible in all of the cases defined above, with exactly one case applying to each member of $B_n(0010)$ or $B_n(0123)$.  Hence, the mapping $h$ provides the desired bijection between $B_n(0010)$ and $B_n(0123)$, which completes the proof.
\end{proof}

We now establish the second entry in Table \ref{tab1} involving the patterns $0011$ and $0112$.

\begin{theorem}\label{0011th}
If $\tau=0011$ or $0112$, then
\begin{equation}\label{0011the1}
f_\tau(x)=\frac{1-4x+5x^2-x^3}{(1-x)(1-2x)^2}.
\end{equation}
\end{theorem}
\begin{proof}
We first find $g=f_{0011}$.  Clearly, members $\pi \in B(0011)$ not containing a repeated letter are enumerated by $\frac{1}{1-x}$, so assume $\pi$ contains a repeated letter.  We consider cases based on the leftmost repeated letter of $\pi$. If $\pi$ starts with two or more $0$'s, then we must have $\pi=0^r10^{u_1}\cdots k0^{u_k}$ for some $k \geq 0$, where $r \geq 2$ and $u_i \geq0$ for all $i \in [k]$. This is seen to yield a contribution towards $g$ of $\frac{x^2}{1-2x}$. Now suppose we have $\pi=01\cdots k0^s\rho$ (*) for some $k,s \geq 1$, where $\rho$ is empty or expressible as $\rho=(k+1)0^{v_1}(k+2)0^{v_2}\cdots$ or $\rho=k0^{v_0}(k+1)0^{v_1}\cdots$, with $v_i\geq0$ for all $i\geq0$. Further, if $\rho$ starts with $k$, then a string $q0^{v_{q-k}}$ for some $q>k$ may be skipped (i.e., there is a single occurrence of consecutive letters in $\rho$ of the form $(q-1)0^a(q+1)0^b$).  This implies $\pi$ of the form (*) are enumerated by
$$\frac{x^3}{(1-x)^2}\left(1+\frac{x}{1-2x}+\frac{x}{1-2x}\cdot\frac{1-x}{1-2x}\right)=\frac{x^3}{(1-2x)^2}.$$

So assume now that the leftmost repeated letter is positive.  If $\pi=01\cdots(k-1)k^r0^s\rho$, where $r \geq 2$, $k,s \geq 1$ and $\rho$ is as before, then one gets the same contribution as from $\pi$ of the form (*) above but multiplied by $\frac{x}{1-x}$ (to account for the additional $r-1$ letters $k$).  This then yields a contribution towards $g$ of $\frac{x^4}{(1-x)(1-2x)^2}$.  Finally, if $\pi$ is of the form $\pi=01\cdots(k-1)k^r\sigma$, where $r \geq 2$ and $\sigma$ is empty or of the form $(k+1)0^{v_1}(k+2)0^{v_2}\cdots$, then one gets a contribution of $\frac{x^3}{(1-x)^2}\left(1+\frac{x}{1-2x}\right)=\frac{x^3}{(1-x)(1-2x)}$. Combining each of the prior cases for $\pi$ yields
\begin{align*}
g&=\frac{1}{1-x}+\frac{x^2}{1-2x}+\frac{x^3}{(1-2x)^2}+\frac{x^4}{(1-x)(1-2x)^2}+\frac{x^3}{(1-x)(1-2x)}=\frac{1-4x+5x^2-x^3}{(1-x)(1-2x)^2},
\end{align*}
as desired.

To determine $f_{0112}$, one can proceed  by considering cases based on the leftmost repeated letter in members of $B(0112)$ to establish an equation for the generating function or make the following observation.
By \cite[Lemma 2.4]{DS}, since $0112$ is contained in the larger pattern $01012$, we have that members of $B_n(0112)$ are synonymous with those in $P_n(0112,0121)$.  The latter pattern pair has already been enumerated and found to have generating function $\frac{1-4x+5x^2-x^3}{(1-x)(1-2x)^2}$; see first Wilf-equivalence class in \cite[Section~4.6]{JMS}.
\end{proof}

Let $w$ be a nonempty nondecreasing integral sequence  with smallest and largest letters $a$ and $b$, respectively.  Then, by the number of \emph{jumps} of $w$, we mean the number of letters in $[a,b]$ not occurring in the range of $w$.  Let $\text{jum}(w)$ denote the number of jumps of $w$ and hence $\text{jum}(w)=b-a-|\text{range}(w)|+1$.  Given $r \geq0$, let $h_r=h_r(x)$ denote the generating function that enumerates nonempty nondecreasing sequences starting with 0 and containing at most $r$ jumps according to the length (marked by $x$).

We will make use of the following formula for $h_r$ in the subsequent proofs of the remaining cases.

\begin{lemma}\label{1000lem}
If $r \geq0$, then $h_r(x)=\frac{x}{1-2x}\left(\frac{x}{1-x}\right)^r$.
\end{lemma}
\begin{proof}
It is instructive for the arguments that follow to provide two proofs of this result.  For the first proof, suppose $w$ is a sequence of length two or more enumerated by $h_r$ where $r \geq 1$ and whose last two letters are $q,p$ for some $0 \leq q \leq p$.  If $p=q$ or $p=q+1$, then the addition of the terminal $p$ does not introduce a jump, and hence the contribution of such $w$ towards the generating function $h_r$ is given by $2xh_r$.  On the other hand, if $q=p+i$ for some $2 \leq i \leq r+1$ and $w=w'p$, then $\text{jum}(w)=\text{jum}(w')+i-1$.  Note $w'$ must then have at most $r-i+1$ jumps, and hence is enumerated by $h_{r-i+1}$.  Thus, considering all possible $i$ yields a contribution towards $h_r$ of $x(h_0+\cdots+h_{r-1})$.  Taking into account the single word of length one enumerated by $h_r$, we obtain from the prior cases the recurrence
\begin{equation}\label{1000leme1}
h_r=x+2xh_r+x(h_0+\cdots+h_{r-1}), \qquad r \geq 1.
\end{equation}
Note that \eqref{1000leme1} also holds when $r=0$, since $h_0=\frac{x}{1-2x}$ as it enumerates members of $\mathcal{T}_0$ according to length.

Let $H(x,y)=\sum_{r\geq0}h_r(x)y^r$.  Multiplying both sides of \eqref{1000leme1} by $y^r$, and summing over all $r \geq0$, implies
\begin{align*}
(1-2x)H(x,y)&=\frac{x}{1-y}+x\sum_{r\geq1}y^r\sum_{i=0}^{r-1}h_i=\frac{x}{1-y}+x\sum_{i\geq0}h_i\sum_{r \geq i+1}y^r=\frac{x}{1-y}+\frac{xy}{1-y}H(x,y).
\end{align*}
Solving for $H(x,y)$ in the last equation gives
$$H(x,y)=\frac{x}{1-2x-y+xy}=\frac{\frac{x}{1-2x}}{1-\frac{(1-x)y}{1-2x}}=\frac{x}{1-2x}\sum_{r \geq0}\left(\frac{1-x}{1-2x}\right)^ry^r,$$
and extracting the coefficient of $y^r$ yields the stated formula for $h_r$.

Our second proof as follows is more enumerative.  Let $\mathcal{A}_n^{(r)}$ for $n \geq 1$ and $r \geq0$ denote the set of $(r+1)$-tuples of words $(w^{(0)},w^{(1)},\ldots,w^{(r)})$ in which $w^{(0)} \in \mathcal{T}_0$ and $w^{(i)} \in \mathcal{T}_0\cup\{\varepsilon\}$ for $1 \leq i \leq r$ such that $\sum_{i=0}^r|w^{(i)}|=n$.
Then it is seen that the coefficient of $x^n$ in $h_r(x)$ gives the cardinality of $\mathcal{A}_n^{(r)}$ for all $n \geq 1$.

Let $\mathcal{B}_n^{(r)}$ denote the set of nondecreasing integral sequences of length $n$ starting with 0 and containing at most $r$ jumps. To complete our second proof, it suffices to define a bijection between  $\mathcal{A}_n^{(r)}$ and  $\mathcal{B}_n^{(r)}$. Let $\lambda=(w^{(0)},\ldots,w^{(r)}) \in \mathcal{A}_n^{(r)}$ and let $1 \leq i_1<\cdots<i_k \leq r$ denote the complete set of indices $i \in [r]$ such that $w^{(i)}$ is nonempty.  Given a word $w$ and a positive integer $m$, let $w+m$ denote the sequence obtained by adding $m$ to each entry of $w$.  Let $\lambda^{(0)}=w^{(0)}$, $\lambda^{(1)}=w^{(i_1)}+m_0+i_1+1$ and $\lambda^{(j)}=w^{(i_j)}+i_j-i_{j-1}+m_{j-1}+1$ for $2 \leq j \leq k$, where $m_j=\max(\lambda^{(j)})$ for $0 \leq j \leq k$ and it is assumed that the $\lambda^{(j)}$ are formed in the natural order starting with $\lambda^{(0)}$. We concatenate the various sequences $\lambda^{(j)}$ and define $\lambda^*=\lambda^{(0)}\lambda^{(1)}\cdots\lambda^{(k)}$. Note that there are $i_j-i_{j-1}$ jumps between $\lambda^{(j)}$ and $\lambda^{(j-1)}$ for $j \in [k]$, where $i_0=0$.  Hence, $\lambda^*$ has $i_1+\sum_{j=2}^k(i_j-i_{j-1})=i_k$ jumps altogether if $k \geq 1$, with no jumps in $\lambda^*$ if no such indices $i_j$ exist.  Thus, $\lambda^*$ contains at most $r$ jumps and hence belongs to $\mathcal{B}_n^{(r)}$.  One may verify that the mapping $\lambda \mapsto \lambda^*$ provides the desired bijection between  $\mathcal{A}_n^{(r)}$ and  $\mathcal{B}_n^{(r)}$.
\end{proof}

\begin{theorem}\label{1000th}
If $\tau=1000$ or $1100$, then
\begin{equation}\label{1000the1}
f_\tau(x)=\frac{(1-x)(1-8x+24x^2-31x^3+13x^4+2x^5+2x^6)}{(1-2x)^5}.
\end{equation}
\end{theorem}
\begin{proof}
To find $g=f_{1000}$, first note that $\pi \in B(1000)$ can contain at most three runs of 0, with at most two 0's occurring beyond the initial run.  If $\pi$ is empty or contains only one run of 0, then $\pi$ avoiding 021 implies it is weakly increasing and hence such $\pi$ are enumerated by $\frac{1-x}{1-2x}$.  If $\pi$ contains exactly two runs of 0, then we have the decomposition
$$\pi=0^{a_0}1^{a_1}\cdots k^{a_k}\delta\pi', \qquad k \geq 1,$$
where $\delta=0$ or $00$, all exponents are positive and $\pi'$ is nondecreasing and starts with $k$ or $k+1$ if nonempty.  If $\pi'$ starts with $k+1$, then it is a staircase word, whereas if $\pi'$ starts with $k$, then it may contain a jump.  Thus, the section $\pi'$ is accounted for by
$$\lambda:=1+\frac{x}{1-2x}+\frac{x}{1-2x}\cdot\frac{1-x}{1-2x}=\left(\frac{1-x}{1-2x}\right)^2,$$
upon considering whether $\pi'$ starts with $k+1$ or $k$.  Since $k \geq 1$, the section $0^{a_0}\cdots k^{a_k}$ of $\pi$ is accounted for by $\frac{x^2}{(1-x)(1-2x)}$.  Thus, by the decomposition above, the $\pi$ in $B(1000)$ containing exactly two runs of 0 make a contribution of $\frac{x^3(1+x)}{(1-x)(1-2x)}\lambda$ towards the generating function $g$.

Now suppose $\pi$ has three runs of 0 and hence $\pi=0^{a_0}\cdots k^{a_k}0\pi'$, where $\pi'$ now contains one zero and does not start with zero.  Let $\pi'=\alpha0\beta$, where $\alpha$ is nonempty and $\beta$ is possibly empty.  If (i) $\alpha \in \mathcal{T}_{k+1}$ or if (ii) $\alpha$ starts with $k$ and contains a jump, then $\beta$ is seen to be accounted for by $\lambda$ in either case.  This implies $\pi'$ for which $\alpha$ satisfies (i) or (ii) is accounted for by $\frac{x^2}{1-2x}\lambda+\frac{x^3}{(1-2x)^2}\lambda=\frac{x^2(1-x)}{(1-2x)^2}\lambda$, and hence the corresponding $\pi$ in $B(1000)$ are enumerated by $\frac{x^5}{(1-2x)^3}\lambda$, upon considering the contribution from the section $0^{a_0}\cdots k^{a_k}0$.

So assume $\alpha \in \mathcal{T}_k$. If the last letter of $\alpha$ is $p$ for some $p \geq k$, then $\beta$ may start with $p$, $p+1$ or $p+2$ if nonempty.  If $\beta$ is empty or starts with $p+1$ or $p+2$, then it is seen that $\beta$ is accounted for by $\lambda$, whereas if $\beta$ starts with $p$, then it is accounted for by $h_2(x)$ whose formula is given in Lemma \ref{1000lem}, as $\beta$ may contain up to two jumps in this case.  This implies $\pi'$ contributes a factor of $\frac{x^2}{1-2x}(\lambda+h_2(x))$, and hence the corresponding $\pi$ in this case are enumerated by $\frac{x^5}{(1-x)(1-2x)^2}(\lambda+h_2(x))$. Combining all of the prior cases implies
\begin{align*}
g&=\frac{1-x}{1-2x}+\frac{x^3(1+x)}{(1-x)(1-2x)}\lambda+\frac{x^5}{(1-2x)^3}\lambda+\frac{x^5}{(1-x)(1-2x)^2}(\lambda+h_2(x))\\
&=\frac{1-x}{1-2x}+\frac{x^3(1-x)}{(1-2x)^3}\left(1+x+\frac{x^2(1-x)}{(1-2x)^2}+\frac{x^2}{1-2x}\right)+\frac{x^6(1-x)}{(1-2x)^5}\\
&=\frac{(1-x)(1-8x+24x^2-31x^3+13x^4+2x^5+2x^6)}{(1-2x)^5},
\end{align*}
as desired.

To complete the proof of \eqref{1000the1}, we define a bijection $f$ between $B_n(1100)$ and $B_n(1000)$ as follows.  Let $\pi \in B_n(1100)$. If $\pi$ does not contain a repeated positive letter, then we must have $\pi=0^{a_0}10^{a_1-1}20^{a_2-1}\cdots k0^{a_k-1}$ for some $k \geq0$, where $a_i>0$ for all $i$, in which case, we let $f(\pi)=0^{a_0}1^{a_1}\cdots k^{a_k}$.  Otherwise, let $j$ denote the smallest repeated positive letter of $\pi$, whence
$\pi=0^{a_0}10^{a_1-1}\cdots j0^{a_{j}-1}j^s\pi'$, where $s \geq 1$ and $\pi'$ does not start with $j$ and contains at most one zero but is otherwise nondecreasing and possibly empty.

First assume $\pi'$ does not contain 0.  Note that if $a_j\geq 2$, then $\pi'$ is either empty, a member of $\mathcal{T}_{j+2}$ or a weakly increasing sequence starting with $j+1$ and containing at most one jump, whereas if $a_j=1$, then $\pi'$ is either empty or a member of $\mathcal{T}_{j+1}$. We then define $f$ based on cases concerning the values of $a_j$ and $s$ as follows:
$$
f(\pi) = \begin{cases}
0^{a_0}1^{a_1}\cdots(j-1)^{a_{j-1}}j^{a_j-1}00j^{s-1}\pi', & \text{if } a_j,s\geq 2; \\
0^{a_0}1^{a_1}\cdots(j-1)^{a_{j-1}}j0j^{a_j-1}\pi', & \text{if }a_j\geq2,\,s=1; \\
0^{a_0}1^{a_1}\cdots(j-1)^{a_{j-1}}j^{s-1}00\pi', & \text{if } a_j=1,\,s\geq 2;\\
0^{a_0}1^{a_1}\cdots(j-1)^{a_{j-1}}j0\pi',  & \text{if } a_j=s=1.
\end{cases}
$$

So assume $\pi'$ contains 0 and we write $$\pi=0^{a_0}10^{a_1-1}\cdots j0^{a_j-1}j^s\sigma0\rho, \quad(*)$$ where $\sigma$ or $\rho$ may be empty and $\sigma$ does not start with $j$.  Note that $\sigma$ if nonempty is either a nondecreasing sequence starting with $j+1$ and containing at most one jump or a member of $\mathcal{T}_{j+2}$,  with $\sigma$ containing a jump or belonging to $\mathcal{T}_{j+2}$ only possible if $a_j\geq 2$.  Further, if $\sigma$ is empty, then $\rho$ can start with $j$, $j+1$ or $j+2$ if nonempty (the $j+2$ option only possible if $a_j \geq 2$).  We define $f$ for $\pi$ of the form (*) as follows.  If it is \emph{not} the case that both $a_j=1$ and $\sigma=\varnothing$,
then let
$$f(\pi)=0^{a_0}1^{a_1}\cdots(j-1)^{a_{j-1}}j^s0j^{a_j-1}\sigma0\rho.$$
If $a_j=1$ and $\sigma=\varnothing$, let
$$f(\pi)=0^{a_0}1^{a_1}\cdots(j-1)^{a_{j-1}}j^{s+1}0\rho.$$

One may verify in each of the last two cases that $f(\pi)$ is indeed an ascent sequence (one that avoids $1000$), and hence $f$ is well-defined.  Further, one can show that $f$ defined in each of the cases above is reversible, with exactly one case of $f$ (or its inverse) applying to any given member of $B_n(1100)$ or $B_n(1000)$.  Thus, $f$ provides the desired bijection between $B_n(1100)$ and $B_n(1000)$, which completes the proof.
\end{proof}

We now consider the pattern $1001$ for which we will make use of an auxiliary parameter as follows.  Given a $021$-avoiding ascent sequence $\pi$, let $\text{pjum}(\pi)=M-m$, where $M=\text{asc}(\pi)$ and $m=\max(\pi)$; note $M \geq m$ for all ascent sequences $\pi$.  Since $\text{pjum}(\pi)$ gives the cardinality of the set $[m+1,M]$, it corresponds to the maximum potential number of jumps (within the subsequence consisting of the positive entries of $\pi$) created when a nonzero letter $s$ is appended to $\pi$ (which occurs with $s=M+1$).  Given $i \geq0$, let $f_i=f_i(x)$ denote the generating function enumerating nonempty members $\pi \in B(1001)$ for which $\text{pjum}(\pi)=i$ according to their length.

We have the following relations involving the generating functions $f_i$.

\begin{lemma}\label{1001lem1}
If $i \geq 1$, then
\begin{equation}\label{1001lem1e1}
f_i=\frac{x^3}{(1-x)^3}f_{i-1}+\frac{x(1-x+x^2)}{(1-x)^3}\sum_{j\geq i}f_j,
\end{equation}
with $f_0=\frac{x}{1-x}+\frac{x}{(1-x)^2}\sum_{j\geq0}f_j$.
\end{lemma}
\begin{proof}
First note that a nonempty $\pi \in B(1001)$ must be of the form $\pi=\pi^{(0)}\pi^{(1)}\cdots \pi^{(k)}$ for some $k \geq 0$, where $\pi^{(0)}$ is a nonempty sequence of $0$'s and $\pi^{(j)}$ for  $j \in [k]$ is given by $\pi^{(j)}=a_j^{\alpha_j}0a_j^{\beta_j}0^{\gamma_j}$ or $a_j^{\alpha_j}0^{\gamma_j}$, with $\alpha_j,\beta_j>0$, $\gamma_j \geq0$ and $1=a_1<\cdots<a_k$.  We will refer to each section $\pi^{(i)}$ as a \emph{unit} of $\pi$.  Conversely, any $\pi$ of the stated form is seen to avoid both $021$ and $1001$, assuming the $a_i$ are such that $\pi$ is indeed an ascent sequence.

We now make some further observations.  Suppose $\text{pjum}(\omega)=i$ and we wish to append a nonzero letter $s$ to $\omega$ such that $\omega s \in B(1001)$ and $s>m$, where $m=\max(\omega)$.  Then the possible values of $s$ comprise the set $[m+1,m+i+1]$, with $\text{pjum}(\omega s)=m+i+1-s$ for each $s$. Further, if the string $s0s$ is appended to $\omega$, where $s$ is as before, then we have $\text{pjum}(\omega s0s)=m+i+2-s$, as the final ascent $0s$ increases the number of ascents by one while maintaining the maximum letter. Note that only when $s=m+1$ and $s0s$ is appended does the resulting ascent sequence have a strictly larger $\text{pjum}$ value than that of $\omega$ (in which case, $\text{pjum}(\omega s0s)=i+1$).  When only $s$ or in all other cases when $s0s$ is appended, the resulting sequence is seen to have pjum value at most $i=\text{pjum}(\omega)$.

Now suppose $\pi$ is enumerated by $f_i$ for some $i \geq1$.  Then $\pi$ must contain at least two units and removing the final unit of $\pi$ results in an ascent sequence $\pi'$ enumerated by $f_j$ for some $j \geq i-1$, by the preceding observations.  Let $r=\max(\pi')$.  If $\pi'$ is enumerated by $f_{i-1}$, then in forming $\pi$ from $\pi'$, one must append a unit of the form $t^u0t^v0^w$, where $t=r+1$, $u,v>0$ and $w \geq0$.  This yields a contribution of $\frac{x^3}{(1-x)^3}f_{i-1}$ towards $f_i$ for such $\pi$, with the $\frac{x^3}{(1-x)^3}$ factor accounting for the appended sequence $t^u0t^v0^w$ wherein the value of $t$ is determined by $\pi'$.  On the other hand, if $\pi'$ is enumerated by $f_j$ for some $j \geq i$, then the final unit of $\pi$ is given by $p^a0^b$, where $a>0$ and $b \geq0$, or by $p^u0p^v0^w$, where $u,v>0$ and $w \geq0$, with $p \geq r+1$ in either case.  Note that the value of $p$ is uniquely determined by $i$ and $\pi'$ (through the values of $r$ and $j$).    This implies such $\pi$ make a contribution towards $f_i$ of
$$\left(\frac{x}{(1-x)^2}+\frac{x^3}{(1-x)^3}\right)f_j=\frac{x(1-x+x^2)}{(1-x)^3}f_j$$
for each $j \geq i$.  Considering the contributions from all possible $j$ yields \eqref{1001lem1e1}.

To establish the formula for $f_0$, first note that $\frac{x}{1-x}$ accounts for those $\pi$ consisting of a single unit, i.e., the all-zero sequences.  So assume $\pi$ contains two or more units and $\text{pjum}(\pi)=0$, where $\pi'$ is as before with $\text{pjum}(\pi')=j$ and $\max(\pi')=r$.  In this case, we must generate $\pi$ from $\pi'$ by appending a unit of the form $p^a0^b$, where $p=r+j+1$, so as to ensure $\text{pjum}(\pi)=0$.  Again, $p$ is determined by $\pi'$, which implies a contribution towards $f_0$ of $\frac{x}{(1-x)^2}f_j$ for each $j \geq0$.  Considering all possible $j$, and combining with the first case above, implies the stated formula for $f_0$ and completes the proof.
\end{proof}

Define the bivariate generating function
$$F(x,y)=\sum_{i\geq0}f_i(x)y^i.$$
Note $f_{1001}=1+F(x,1)$, by the definitions.

Multiplying both sides of \eqref{1001lem1e1} by $y^i$, summing over all $i \geq 1$ and taking into account the formula for $f_0(x)$, we have
\begin{align*}
F(x,y)&=\frac{x}{1-x}+\frac{x}{(1-x)^2}\sum_{j\geq0}f_j+\frac{x^3}{(1-x)^3}\sum_{i\geq1}f_{i-1}y^i+\frac{x(1-x+x^2)}{(1-x)^3}\sum_{j\geq1}f_j\sum_{i=1}^jy^i\\
&=\frac{x}{1-x}+\frac{x}{(1-x)^2}F(x,1)+\frac{x^3y}{(1-x)^3}F(x,y)+\frac{x(1-x+x^2)}{(1-x)^3}\sum_{j\geq0}f_j\cdot\frac{y-y^{j+1}}{1-y}\\
&=\frac{x}{1-x}+\frac{x}{(1-x)^2}F(x,1)+\frac{x^3y}{(1-x)^3}F(x,y)+\frac{xy(1-x+x^2)}{(1-x)^3(1-y)}(F(x,1)-F(x,y)).
\end{align*}
Solving for $F(x,y)$ in the last equality yields the following functional equation.

\begin{lemma}\label{1001lem2}
We have
\begin{equation}\label{1001lem2e1}
\left(1-\frac{x^3y}{(1-x)^3}+\frac{xy(1-x+x^2)}{(1-x)^3(1-y)}\right)F(x,y)=\frac{x}{1-x}+\left(\frac{x}{(1-x)^2}+\frac{xy(1-x+x^2)}{(1-x)^3(1-y)}\right)F(x,1).
\end{equation}
\end{lemma}

We can now find an explicit formula for $f_{1001}(x)$.

\begin{theorem}\label{1001th}
If $\tau\in\{1001,1011,1101\}$, then
\begin{equation}\label{1001the1}
f_\tau(x)=\frac{1-3x+3x^2-\sqrt{(1-3x+x^2)^2-4x^3(1-x)}}{2x^2}.
\end{equation}
\end{theorem}
\begin{proof}
We begin by defining bijections between $B_n(1001)$ and $B_n(\tau)$ for the other two patterns $\tau$.  A bijection between $B_n(1001)$ and $B_n(1011)$ may be realized by changing each unit of $\pi$ of the form $u=i^a0i^b0^c$, where $a,b,i>0$ and $c\geq0$, to $i^a0^bi0^c$ and leaving all other units unchanged.  For a bijection with $B_n(1101)$, we instead change $u$ to $i0^ai^b0^c$ for each such unit $u$.

Thus, we need only compute $f_{1001}$.  To do so, we apply the \emph{kernel method} (see, e.g., \cite{HM}) to \eqref{1001lem2e1} and set the factor multiplying $F(x,y)$ on the left-hand side equal to zero.  This gives
\begin{equation}\label{1001the2}
x^3y^2-(1-x)(1-3x+x^2)y+(1-x)^3=0.
\end{equation}
We denote the solution in $y$ to \eqref{1001the2} by $\widetilde{y}$, where we choose the negative root (as the positive root is not seen to lead to a power series expression for $F(x,1)$).  Note
$$\widetilde{y}=\frac{(1-x)(1-3x+x^2)-(1-x)\sqrt{(1-3x+x^2)^2-4x^3(1-x)}}{2x^3}.$$
Substituting $y=\widetilde{y}$ in \eqref{1001lem2e1}, and observing $\frac{x\widetilde{y}(1-x+x^2)}{(1-x)^3(1-\widetilde{y})}=\frac{x^3\widetilde{y}}{(1-x)^3}-1,$
yields
$$\left(\frac{x}{(1-x)^2}+\frac{x^3\widetilde{y}}{(1-x)^3}-1\right)F(x,1)+\frac{x}{1-x}=0.$$
Thus, we have
\begin{align*}
F(x,1)&=\frac{-x(1-x)^2}{x(1-x)-(1-x)^3+x^3\widetilde{y}}=\frac{2x(1-x)}{1-3x+x^2+\sqrt{(1-3x+x^2)^2-4x^3(1-x)}}\\
&=\frac{1-3x+x^2-\sqrt{(1-3x+x^2)^2-4x^3(1-x)}}{2x^2}.
\end{align*}
The desired formula now follows from the fact $f_{1001}=1+F(x,1)$, which completes the proof.
\end{proof}

\noindent \emph{Remark:} Substituting the formula found above for $F(x,1)$ back into \eqref{1001lem2e1} leads to an explicit expression for $F(x,y)$, whose coefficient of $x^n$ for $n \geq 1$ is seen to be the distribution on $B_n(1001)$ of the pjum statistic (marked by $y$).

\begin{theorem}\label{1010th}
If $\tau=1010$ or $1120$, then
\begin{equation}\label{1010the1}
f_\tau(x)=\frac{(1-x)(1-5x+7x^2-x^3)}{(1-2x)^2(1-3x+x^2)}.
\end{equation}
\end{theorem}
\begin{proof}
We first consider the case $\tau=1010$.  Let $B^*$ denote the subset of $B(1010)$ whose members have no equal adjacent letters, i.e., contain no levels.  To find $f_{1010}$, it is more convenient to first find its restriction to $B^*$, which we will denote by $g$, and then replace $x$ by $\frac{x}{1-x}$ in $g$.  Note that one may proceed in this manner since the pattern in question contains no levels.  Clearly, members of $B^*$ of the form $01\cdots k$ for some $k$, including $\varepsilon$, are enumerated by $\frac{1}{1-x}$.  So assume $\pi \in B^*$ contains a repeated letter, the first of which must be zero.  That is, $\pi=01\cdots k0\pi'$ for some $k \geq 1$, where $\pi'$ is possibly empty.  If $\pi'=\varepsilon$ or $\pi'$ starts with $k+1$, then the section $0\pi'$ is synonymous with a nonempty member of $B^*$, upon subtracting $k$ from each nonzero letter in $0\pi'$.  Taking into account the initial string $01\cdots k$ for some $k \geq 1$, we get a contribution towards $g$ of $\frac{x^2}{1-x}(g-1)$ from such $\pi$.

On the other hand, if $\pi'$ starts with $k$, then $\pi'$ cannot contain 0 in order to avoid $1010$ and hence $\pi'=k(k+1)\cdots$ or $\pi'=k(k+1)\cdots(\ell-1)(\ell+1)(\ell+2)\cdots$ for some $\ell>k$. This implies that members of $B^*$ of the form $01\cdots k0\pi'$ where $\pi'$ starts with $k$ are enumerated by
$$\frac{x^3}{1-x}\cdot\frac{x}{1-x}+\frac{x^3}{1-x}\cdot\left(\frac{x}{1-x}\right)^2=\frac{x^4}{(1-x)^3}.$$
Combining the prior cases, we then get
$$g=\frac{1}{1-x}+\frac{x^2}{1-x}(g-1)+\frac{x^4}{(1-x)^3},$$
and solving for $g$ gives $g=\frac{1-2x+2x^3}{(1-x)^2(1-x-x^2)}$.
Thus, we have
$$f_{1010}(x)=g\left(\frac{x}{1-x}\right)=\frac{(1-x)(1-5x+7x^2-x^3)}{(1-2x)^2(1-3x+x^2)}.$$

We now compute $f_{1120}$.  The weight of the all-zero sequences, together with $\varepsilon$, is given by $\frac{1}{1-x}$, whereas those that contain 1 but have no repeated positive letter are seen to contribute $\frac{x^2}{(1-x)(1-2x)}$.  Now suppose $\pi \in B(1120)$ is such that its largest letter is the only repeated positive letter.  Then $\pi$ may be decomposed as
$$\pi=0^{u_0}10^{u_1}\cdots(k-1)0^{u_{k-1}}\pi', \qquad k \geq 1,$$
where $u_0>0$, $u_i\geq0$ for $1 \leq i \leq k-1$ and $\pi'$ starts with $k$ and contains at least two $k$'s.  Note that $\pi'$ is `binary' on $\{0,k\}$ in order to avoid 021, and hence is accounted for by
$\frac{x}{1-2x}-\frac{x}{1-x}=\frac{x^2}{(1-x)(1-2x)}$,
upon subtracting the weight of sequences of the form $k0^u$ for some $u\geq0$ since $k$ is required to occur at least twice.  The section $0^{u_0}10^{u_1}\cdots(k-1)0^{u_{k-1}}$, with the $u_i$ as given, contribute a factor of $\frac{x}{1-x}\left(1+\frac{x}{1-2x}\right)=\frac{x}{1-2x}$, and hence $\pi$ of the stated form are enumerated by $\frac{x}{1-2x}\cdot\frac{x^2}{(1-x)(1-2x)}=\frac{x^3}{(1-x)(1-2x)^2}$.

Next, we assume $\pi \in B(1120)$ contains a repeated positive letter that is not the largest letter.  Then we have for some $k, r \geq 1$ the decomposition
$$\pi=0^{u_0}10^{u_1}\cdots(k-1)0^{u_{k-1}}k^{v_1}0^{w_1}\cdots k^{v_r}0^{w_r}\pi', \quad (*)$$
where $\pi'$ is nonempty with first letter greater than $k$, the $u_i$ are as before, $v_1,\ldots,v_r>0$, $v_1+\cdots+v_r\geq 2$, $w_1,\ldots,w_{r-1}>0$ and $w_r \geq0$.  Then the section (which we will denote by $\pi-\pi'$) preceding $\pi'$ has $k+r-1$ ascents and thus the first letter $p$ of $\pi'$ satisfies $k+1 \leq p \leq k+r$, with $\pi'$ nondecreasing.  Note that if $\pi'$ starts with $p$, then $\pi'$ can contain at most $k+r-p$ jumps, with $k$ and $r$ being determined by $\pi-\pi'$.  Thus, if $\pi-\pi'$ is specified, along with the desired maximum possible number $s$ of jumps in $\pi'$ (with $s \in \{0,1,\ldots,r-1\}$), then the first letter of $\pi'$ is uniquely determined (and in fact is given by $k+r-s$).  Thus, we have that the section $\pi'$ within $\pi$ of the form (*) is accounted for by $\sum_{j=0}^{r-1}h_j(x)$ once $\pi-\pi'$ has been specified.

To determine the contribution towards $f_{1120}$ coming from $\pi$ of the form (*), we must differentiate the cases when $r=1$ and $r \geq 2$.  If $r=1$, then we subtract the excluded sequences of the form $k0^w$ for some $w \geq0$ as $k$ is to occur at least twice.  Hence, the string $k^{v_1}0^{w_1}$ appearing in $\pi$ of the form (*) where $r=1$ is accounted for by $\frac{x}{(1-x)^2}-\frac{x}{1-x}=\frac{x^2}{(1-x)^2}$.  Thus, the $\pi$ of the form (*) where $r=1$ are enumerated by
$$\frac{x}{1-2x}\cdot\frac{x^2}{(1-x)^2}\cdot h_0(x)=\frac{x^4}{(1-x)^2(1-2x)^2}.$$
On the other hand, by the previous observations, we have that $\pi$ of the form (*) where $r \geq 2$ are enumerated by
$$\frac{x}{1-2x}\cdot \frac{x^{2r-1}}{(1-x)^{2r}}\cdot\sum_{j=0}^{r-1}h_j(x)$$
for each $r$.  Summing over all $r \geq 2$, and making use of Lemma \ref{1000lem}, then gives
\begin{align*}
&\frac{1}{1-2x}\sum_{r\geq2}\left(\frac{x}{1-x}\right)^{2r}\sum_{j=0}^{r-1}h_j(x)=\frac{x}{(1-2x)^2}\sum_{r\geq2}\left(\frac{x}{1-x}\right)^{2r}\sum_{j=0}^{r-1}\left(\frac{1-x}{1-2x}\right)^j\\
&=\frac{x}{(1-2x)^2}\sum_{r\geq2}\left(\frac{x}{1-x}\right)^{2r}\cdot\frac{1-\left(\frac{1-x}{1-2x}\right)^r}{1-\frac{1-x}{1-2x}}=\frac{1}{1-2x}\sum_{r\geq2}\left(\frac{x}{1-x}\right)^{2r}\left(\left(\frac{1-x}{1-2x}\right)^r-1\right)\\
&=\frac{1}{1-2x}\sum_{r\geq2}\left(\left(\frac{x^2}{(1-x)(1-2x)}\right)^r-\left(\frac{x}{1-x}\right)^{2r}\right)=\frac{x^5(2-x)}{(1-x)^2(1-2x)^2(1-3x+x^2)}.
\end{align*}
Combining all of the prior cases for $\pi \in B(1120)$, we get
\begin{align*}
f_{1120}&=\frac{1}{1-x}+\frac{x^2}{(1-x)(1-2x)}+\frac{x^3}{(1-x)(1-2x)^2}+\frac{x^4}{(1-x)^2(1-2x)^2}+\frac{x^5(2-x)}{(1-x)^2(1-2x)^2(1-3x+x^2)}\\
&=\frac{(1-x)(1-5x+7x^2-x^3)}{(1-2x)^2(1-3x+x^2)},
\end{align*}
which completes the proof.
\end{proof}

\noindent\emph{Remarks:} Several of the underlying sequences from the prior results correspond to entries in the OEIS \cite{Sloane} and yield new combinatorial interpretations of these entries in terms of pattern-avoiding ascent sequences.   We first note that the sequence $b_n(1010)$ or $b_n(1120)$ for $n \geq0$ from the preceding theorem corresponds to entry A244885 in the OEIS.  This sequence is also known to enumerate other combinatorial structures, such as the $1243$-avoiding Catalan words of length $n$ or the peak-equivalence classes on the set of Dyck paths of semi-length $n$, and it would be interesting to find a bijection with such a structure and $B_n(1010)$ or $B_n(1120)$.  From Theorem \ref{0011th}, it is seen that $b_n(0011)=b_n(0112)=(n-1)2^{n-2}+1$ for $n \geq 1$, which matches entry A005183 with offset one.  Finally, the sequence from Theorem \ref{1001th} matches that of A365510 for $1 \leq n \leq 11$, though no formula of the generating function is given for the terms in this entry.  Perhaps it would be possible to find an explicit bijection between $B_n(\tau)$ where $\tau \in \{1001,1011,1101\}$ and one of the known structures enumerated by A365510. \medskip

\begin{theorem}\label{1020th}
If $\tau=1020$ or $1022$, then
\begin{equation}\label{1020the1}
f_\tau(x)=\frac{1-9x+32x^2-56x^3+49x^4-19x^5+x^6}{(1-x)(1-2x)^3(1-3x+x^2)}.
\end{equation}
\end{theorem}
\begin{proof}
We first find $f_{1020}$.  Clearly, the binary members of $B(1020)$, together with $\varepsilon$, are enumerated by $\frac{1-x}{1-2x}$.  So assume henceforth at $\pi \in B(1020)$ contains at least one letter greater than 1.  Suppose first that at least one 0 occurs between the first 1 and the leftmost letter exceeding 1.  In this case, we have
$$\pi=0^{a_0}1^{b_1}0^{a_1}\cdots1^{b_\ell}0^{a_\ell}j\pi', \qquad \ell \geq 1,$$
where $\pi'$ is possibly empty, $2 \leq j \leq \ell+1$, $a_i>0$ for $0 \leq i \leq \ell-1$, $b_i>0$ for all $i$ and $a_\ell \geq0$ if $\ell \geq 2$, with $a_1>0$ if $\ell=1$.  Note that the section $j\pi'$ cannot contain 0 in order to avoid 1020, and hence is nondecreasing with at most $\ell+1-j$ jumps.  Thus, it is accounted for by $h_{\ell+1-j}$, where $h_i$ is as in Lemma \ref{1000lem}, for each $\ell$ and $j$.   Allowing $j$ to vary then yields a contribution of $\sum_{j=2}^{\ell+1}h_{\ell+1-j}$ for the section $j\pi'$ for each $\ell$.  Considering separately the cases when $\ell=1$ or $\ell \geq 2$, we have that $\pi$ of the stated form above are enumerated by
\begin{align*}
&\frac{x^3}{(1-x)^3}h_0+\sum_{\ell \geq 2}\frac{x^{2\ell}}{(1-x)^{2\ell+1}}\sum_{j=2}^{\ell+1}h_{\ell+1-j}=\frac{x^4}{(1-x)^3(1-2x)}+\frac{1}{1-2x}\sum_{\ell \geq2}\left(\frac{x}{1-x}\right)^{2\ell+1}\cdot\frac{1-\left(\frac{1-x}{1-2x}\right)^\ell}{1-\frac{1-x}{1-2x}}\\
&=\frac{x^4}{(1-x)^3(1-2x)}+\frac{1}{1-x}\sum_{\ell \geq2}\left(\left(\frac{x^2}{(1-x)(1-2x)}\right)^\ell-\left(\frac{x}{1-x}\right)^{2\ell}\right)\\
&=\frac{x^4}{(1-x)^3(1-2x)}+\frac{x^5(2-x)}{(1-x)^3(1-2x)(1-3x+x^2)}=\frac{x^4}{(1-x)^2(1-2x)(1-3x+x^2)}.
\end{align*}

So assume $\pi$ does not contain a 0 between its first 1 and its leftmost letter greater than 1, i.e., $\pi$ starts $0^{a_0}1^{a_1}2^{a_2}$.  If $\pi$ is nondecreasing, then such $\pi$ contribute $\sum_{k\geq2}\left(\frac{x}{1-x}\right)^{k+1}=\frac{x^3}{(1-x)^2(1-2x)}$ towards $f_{1020}$.  Otherwise, we must have $\pi=0^{a_0}1^{a_1}\cdots k^{a_k}0^r\pi'$, where $k \geq 2$, all exponents are positive and $\pi'$ starts with $k$ or $k+1$ if nonempty.  It is seen for each given $k$ that the section $\pi'$ contributes a factor of $\lambda$, where $\lambda$ is as in the proof of Theorem \ref{1000th}.  Since the remaining letters $0^{a_0}1^{a_1}\cdots k^{a_k}0^r$ for some $k \geq 2$ are accounted for by $\frac{x^4}{(1-x)^3(1-2x)}$, we have that $\pi$ of the stated form are enumerated by $\frac{x^4}{(1-x)^3(1-2x)}\lambda=\frac{x^4}{(1-x)(1-2x)^3}$.
Combining all of the cases of $\pi$ above then yields
\begin{align*}
f_{1020}&=\frac{1-x}{1-2x}+\frac{x^4}{(1-x)^2(1-2x)(1-3x+x^2)}+\frac{x^3}{(1-x)^2(1-2x)}+\frac{x^4}{(1-x)(1-2x)^3}\\
&=\frac{1-9x+32x^2-56x^3+49x^4-19x^5+x^6}{(1-x)(1-2x)^3(1-3x+x^2)}.
\end{align*}

To complete the proof of \eqref{1020the1}, we define a bijection $g$ between $B_n(1020)$ and $B_n(1022)$.  Let $\rho \in B_n(1020)$ and consider the section $\rho'$, if it exists, occurring to the right of the rightmost 0 of $\rho$, assuming $\rho$ has at least two runs of $0$.  If no such section $\rho'$ exists, then let $g(\rho)=\rho$.  Otherwise, replace each (maximal) string within $\rho'$ of the form $a^t$, where $a \geq 2$ and $t \geq 1$, with $a0^{t-1}$. Let $g(\rho)$ denote the resulting ascent sequence.  Combining the two cases, one may verify that $g$ provides the desired bijection between $B_n(1020)$ and $B_n(1022)$.
\end{proof}

\begin{theorem}\label{1200th}
If $\tau \in \{1200,1220,1230\}$, then
\begin{equation}\label{1200the1}
f_\tau(x)=\frac{1-7x+17x^2-16x^3+5x^4-x^5}{(1-2x)(1-3x+x^2)^2}.
\end{equation}
\end{theorem}
\begin{proof}
We first find $f_{1230}$. The binary members, together with $\varepsilon$, are accounted for by $\frac{1-x}{1-2x}$, so assume henceforth $\pi \in B(1230)$ contains a letter great than 1.  First suppose $\pi$ can be decomposed as
$$\pi=0^{a_0}1^{b_1}0^{a_1}\cdots 1^{b_\ell}0^{a_\ell}\pi', \qquad \ell \geq 1,$$
where $a_i>0$ for $0 \leq i \leq \ell-1$ with $a _\ell \geq0$, $b_i>0$ for all $i$ and $\pi'$ is a nonempty sequence on $\{0,u\}$ for some $u>1$ that starts with $u$.  Note $\pi$ an ascent sequence implies $2 \leq u \leq \ell+1$. Considering all possible $\ell$, we have that $\pi$ of the stated form are enumerated by
\begin{align*}
\sum_{\ell \geq 1}\frac{x^{2\ell}}{(1-x)^{2\ell+1}}\cdot\frac{\ell x}{1-2x}&=\frac{x^3}{(1-x)^3(1-2x)}\sum_{\ell \geq1}\ell\left(\frac{x^2}{(1-x)^2}\right)^{\ell-1}=\frac{x^3}{(1-x)^3(1-2x)}\cdot\frac{1}{\left(1-\frac{x^2}{(1-x)^2}\right)^2}\\
&=\frac{x^3(1-x)}{(1-2x)^3}.
\end{align*}

So assume $\pi$ contains at least three distinct positive letters.  Then we have
$$\pi=0^{a_0}1^{b_1}0^{a_1}\cdots1^{b_\ell}0^{a_{\ell}}u^{c_1}0^{d_1}\cdots u^{c_p}0^{d_p}\pi', \qquad \ell, p \geq 1,$$
where the $a_i$, $b_i$ and $u$ are as before, $c_i>0$ for all $i$, $d_i>0$ for $1 \leq i \leq p-1$ with $d_p\geq0$ and $\pi' \neq \varepsilon$ starts with $v$ for some $v>u$.  Note that there are $\ell+p$ ascents occurring strictly prior to the first $v$ and hence $v \in [u+1,\ell+p+1]$.  Then $\pi$ avoiding $021$ and $1230$ implies $\pi'$ is nondecreasing and hence is accounted for by $h_{\ell+p+1-v}(x)$, as it contains at most $\ell+p+1-v$ jumps, where $h_i(x)$ is as in Lemma \ref{1000lem}.  Conversely, any $\pi$ of the stated form with $\pi'$ as described is seen to belong to $B(1230)$. Considering all possible $\ell$, $p$, $u$ and $v$, we have that $\pi$ of the form above are enumerated by
$$\sum_{\ell \geq1}\sum_{p \geq1}\sum_{u=2}^{\ell+1}\sum_{v=u+1}^{\ell+p+1}\frac{x^{2\ell}}{(1-x)^{2\ell+1}}\cdot\frac{x^{2p-1}}{(1-x)^{2p}}\cdot h_{\ell+p+1-v}(x).$$

This last sum may be simplified as follows:
\begin{align*}
&\sum_{\ell \geq1}\sum_{p \geq1}\sum_{u=2}^{\ell+1}\sum_{v=u+1}^{\ell+p+1}\frac{x^{2\ell+2p-1}}{(1-x)^{2\ell+2p+1}}\cdot\frac{x}{1-2x}\left(\frac{1-x}{1-2x}\right)^{\ell+p+1-v}\\
&=\frac{1}{(1-2x)^2}\sum_{\ell \geq1}\sum_{p \geq1}\left(\frac{x^2}{(1-x)(1-2x)}\right)^{\ell+p}\sum_{u=2}^{\ell+1}\sum_{v=u+1}^{\ell+p+1}\left(\frac{1-2x}{1-x}\right)^v\\
&=\frac{1}{x(1-2x)}\sum_{\ell \geq1}\sum_{p \geq1}\left(\frac{x^2}{(1-x)(1-2x)}\right)^{\ell+p}\sum_{u=2}^{\ell+1}\left(\left(\frac{1-2x}{1-x}\right)^u-\left(\frac{1-2x}{1-x}\right)^{\ell+p+1}\right)\qquad\qquad\qquad\qquad\quad\\
&=\frac{1}{x(1-2x)}\sum_{\ell \geq1}\sum_{p \geq1}\left(\frac{x^2}{(1-x)(1-2x)}\right)^{\ell+p}\left(\frac{\left(\frac{1-2x}{1-x}\right)^2-\left(\frac{1-2x}{1-x}\right)^{\ell+2}}{1-\frac{1-2x}{1-x}}-\ell\left(\frac{1-2x}{1-x}\right)^{\ell+p+1}\right)\\
&=\frac{1-2x}{x^2(1-x)}\sum_{\ell \geq1}\sum_{p \geq1}\left(\frac{x^2}{(1-x)(1-2x)}\right)^{\ell+p}\left(1-\left(\frac{1-2x}{1-x}\right)^\ell\right)-\frac{1}{x(1-x)}\sum_{\ell\geq1}\sum_{p\geq1}\ell\left(\frac{x^2}{(1-x)^2}\right)^{\ell+p}\\
&=\frac{1-2x}{x^2(1-x)}\cdot\frac{x^4}{(1-3x+x^2)^2}-\frac{1-2x}{x^2(1-x)}\cdot\frac{x^4}{(1-2x)(1-3x+x^2)}-\frac{1}{x(1-x)}\cdot\frac{x^4(1-x)^2}{(1-2x)^3}\\
&=\frac{x^3}{(1-3x+x^2)^2}-\frac{x^3(1-x)}{(1-2x)^3}.
\end{align*}
Combining the last case with the two prior implies
$$f_{1230}=\frac{1-x}{1-2x}+\frac{x^3}{(1-3x+x^2)^2}=\frac{1-7x+17x^2-16x^3+5x^4-x^5}{(1-2x)(1-3x+x^2)^2},$$
as desired.

We now find $f_{1200}$. The binary sequences are accounted for by $\frac{1-x}{1-2x}$, so assume $\pi \in B(1200)$ contains a letter greater than 1.  First suppose $\pi$ is decomposable as
\begin{equation}\label{1200form1}
\pi=0^{a_0}1^{b_1}0^{a_1}\cdots1^{b_\ell}0^{a_\ell}\pi'\sigma, \qquad \ell \geq 1,
\end{equation}
where $a_j>0$ for $0 \leq j \leq \ell-1$ with $a_\ell \geq0$, $b_j>0$ for all $j$, $\sigma=0$ or $\varepsilon$, and $\pi' \neq \varepsilon$ starts with $i \in [2,\ell+1]$ and does not contain 0. Note $\pi'$ is nondecreasing with at most $\ell+1-i$ jumps, and hence is enumerated by $h_{\ell+1-i}$. Considering all possible $\ell$ and $i$ then implies $\pi$ of the form \eqref{1200form1} are enumerated by
\begin{align*}
&\sum_{\ell \geq1}\frac{x^{2\ell}(1+x)}{(1-x)^{2\ell+1}}\sum_{i=2}^{\ell+1}\frac{x}{1-2x}\left(\frac{1-x}{1-2x}\right)^{\ell+1-i}=\frac{1+x}{1-x}\sum_{\ell \geq1}\left(\frac{x}{1-x}\right)^{2\ell}\left(\left(\frac{1-x}{1-2x}\right)^\ell-1\right)\\
&=\frac{x^2(1+x)}{1-x}\left(\frac{1}{1-3x+x^2}-\frac{1}{1-2x}\right)=\frac{x^3(1+x)}{(1-2x)(1-3x+x^2)}.
\end{align*}

Otherwise, $\pi$ is expressible as
\begin{equation}\label{1200form2}
\pi=0^{a_0}1^{b_1}0^{a_1}\cdots1^{b_\ell}0^{a_\ell}\pi'0\pi'', \qquad \ell \geq1,
\end{equation}
where the $a_i$ and $b_i$ are as before, $\pi'$ and $\pi''$ are nonempty and $\pi'$ starts with $i$ and has exactly $j$ jumps for some $i \in [2,\ell+1]$ and $0 \leq j \leq \ell+1-i$.  Note $\pi$ avoiding $1200$ implies neither $\pi'$ nor $\pi''$ can contain 0, and hence are nondecreasing.
If $p$ denotes the last letter of $\pi'$ and $q$ the first letter of $\pi''$, then we have $q=p+a$ for some $0 \leq a \leq \ell+2-i-j$, as $\pi'$ starting with $i$ and containing exactly $j$ jumps implies that at most $\ell+1-i-j$ integers can be skipped in going from $p$ to $q$.  Then $\pi''$ starting with $p=q+a$ implies it can contain at most $\ell+2-i-j-a$ jumps for each $a$.

Let $h_j^*=h_j^*(x)$ denote the generating function that counts nonempty nondecreasing sequences starting with 0 and containing exactly $j$ jumps.  Then we have $h_0^*=h_0$ and $h_j^*=h_j-h_{j-1}$ for $j \geq 1$, by the definitions.  By Lemma \ref{1000lem}, we have
$$h_j^*(x)=\frac{x}{1-2x}\left(\left(\frac{1-x}{1-2x}\right)^j-\left(\frac{1-x}{1-2x}\right)^{j-1}\right)=\frac{x^2(1-x)^{j-1}}{(1-2x)^{j+1}}, \qquad j \geq 1.$$
Considering all possible $\ell$, $i$, $j$ and $a$ implies $\pi$ of the form \eqref{1200form2} are enumerated by
$$\sum_{\ell \geq1}\frac{x^{2\ell}}{(1-x)^{2\ell+1}}\sum_{i=2}^{\ell+1}\sum_{a=0}^{\ell+2-i}xh_0(x)h_{\ell+2-i-a}(x)+\sum_{\ell \geq2}\frac{x^{2\ell}}{(1-x)^{2\ell+1}}\sum_{i=2}^{\ell}\sum_{j=1}^{\ell+1-i}h_j^*(x)\sum_{a=0}^{\ell+2-i-j}xh_{\ell+2-i-j-a}(x),$$
where we have differentiated the cases when $j=0$ or $j \geq 1$.  The first multi-sum in the preceding expression may be simplified as follows:
\begin{align*}
&\frac{x^3}{(1-x)(1-2x)^2}\sum_{\ell \geq1}\left(\frac{x}{1-x}\right)^{2\ell}\sum_{i=2}^{\ell+1}\sum_{a=0}^{\ell+2-i}\left(\frac{1-x}{1-2x}\right)^{\ell+2-i-a}\\
&=\frac{x^2}{(1-x)(1-2x)}\sum_{\ell \geq1}\left(\frac{x}{1-x}\right)^{2\ell}\sum_{i=2}^{\ell+1}\left(\left(\frac{1-x}{1-2x}\right)^{\ell+3-i}-1\right)\\
&=\frac{x(1-x)}{(1-2x)^2}\sum_{\ell \geq1}\left(\frac{x^2}{(1-x)(1-2x)}\right)^{\ell}\left(1-\left(\frac{1-2x}{1-x}\right)^\ell\right)-\frac{x^4}{(1-x)^3(1-2x)}\sum_{\ell\geq1}\ell\left(\frac{x^2}{(1-x)^2}\right)^{\ell-1}\\
&=\frac{x^4(1-x)^2}{(1-2x)^3(1-3x+x^2)}-\frac{x^4(1-x)}{(1-2x)^3}.
\end{align*}

The second multi-sum above is given by
\begin{align}
&\sum_{\ell \geq2}\frac{x^{2\ell}}{(1-x)^{2\ell+1}}\sum_{i=2}^\ell\sum_{j=1}^{\ell+1-i}\frac{x^2(1-x)^{j-1}}{(1-2x)^{j+1}}\sum_{a=0}^{\ell+2-i-j}\frac{x^2}{1-2x}\left(\frac{1-x}{1-2x}\right)^{\ell+2-i-j-a}\notag\\
&=\frac{x^4}{(1-2x)^4}\sum_{\ell \geq2}\left(\frac{x^2}{(1-x)(1-2x)}\right)^\ell\sum_{i=2}^\ell\sum_{j=1}^{\ell+1-i}\sum_{a=0}^{\ell+2-i-j}\left(\frac{1-2x}{1-x}\right)^{i+a}.\label{multsum2*}
\end{align}
Simplifying the inner triple sum in the last expression, we have
\begin{align*}
&\sum_{i=2}^\ell\sum_{j=1}^{\ell+1-i}\sum_{a=0}^{\ell+2-i-j}\left(\frac{1-2x}{1-x}\right)^{i+a}=\frac{1-x}{x}\sum_{i=2}^\ell\left(\frac{1-2x}{1-x}\right)^i\sum_{j=1}^{\ell+1-i}\left(1-\left(\frac{1-2x}{1-x}\right)^{\ell+3-i-j}\right)\\
&=\frac{1-x}{x}\sum_{i=2}^{\ell}(\ell+1-i)\left(\frac{1-2x}{1-x}\right)^i-\frac{1-x}{x}\sum_{i=2}^\ell\left(\frac{1-2x}{1-x}\right)^{\ell+3}\sum_{j=1}^{\ell+1-i}\left(\frac{1-x}{1-2x}\right)^j\\
&=\frac{1-x}{x}\sum_{i=1}^{\ell-1}i\left(\frac{1-2x}{1-x}\right)^{\ell+1-i}+\frac{(1-x)^2}{x^2}\sum_{i=2}^{\ell+1}\left(\left(\frac{1-2x}{1-x}\right)^{\ell+3}-\left(\frac{1-2x}{1-x}\right)^{i+2}\right).
\end{align*}
Note that the last formula yields zero when $\ell=1$, and hence one may start the outer sum in $\ell$ on the right side of \eqref{multsum2*} at $\ell=1$.  Substituting this last expression for the inner triple sum into \eqref{multsum2*}, we get
\begin{align*}
&\frac{x^3(1-x)}{(1-2x)^4}\sum_{\ell \geq2}\left(\frac{x}{1-x}\right)^{2\ell}\sum_{i=1}^{\ell-1}i\left(\frac{1-x}{1-2x}\right)^{i-1}+\frac{x^2}{(1-x)(1-2x)}\sum_{\ell\geq1}\ell\left(\frac{x}{1-x}\right)^{2\ell}\\
&\quad-\frac{x^2}{(1-x)^2}\sum_{\ell\geq1}\left(\frac{x^2}{(1-x)(1-2x)}\right)^\ell\sum_{i=2}^{\ell+1}\left(\frac{1-2x}{1-x}\right)^{i-2}\\
&=\frac{x^3(1-x)}{(1-2x)^4}\sum_{i\geq1}i\left(\frac{1-x}{1-2x}\right)^{i-1}\cdot\frac{(x^2/(1-x)^2)^{i+1}}{1-\frac{x^2}{(1-x)^2}}+\frac{x^2}{(1-x)(1-2x)}\cdot\frac{x^2}{(1-x)^2\left(1-\frac{x^2}{(1-x)^2}\right)^2}\\
&\quad-\frac{x}{1-x}\sum_{\ell \geq1}\left(\frac{x^2}{(1-x)(1-2x)}\right)^\ell\left(1-\left(\frac{1-2x}{1-x}\right)^\ell\right)\\
&=\frac{x^7}{(1-x)(1-2x)^5}\sum_{i\geq1}i\left(\frac{x^2}{(1-x)(1-2x)}\right)^{i-1}+\frac{x^4(1-x)}{(1-2x)^3}-\frac{x^3}{1-x}\left(\frac{1}{1-3x+x^2}-\frac{1}{1-2x}\right)\\
&=\frac{x^7(1-x)}{(1-2x)^3(1-3x+x^2)^2}+\frac{x^4(1-x)}{(1-2x)^3}-\frac{x^4}{(1-2x)(1-3x+x^2)}.
\end{align*}

Thus, from the expressions found above for the two multi-sums, we have that $\pi$ of the form \eqref{1200form2} are enumerated by
$$\frac{x^4(1-x)^2}{(1-2x)^3(1-3x+x^2)}+\frac{x^7(1-x)}{(1-2x)^3(1-3x+x^2)^2}-\frac{x^4}{(1-2x)(1-3x+x^2)}.$$
Combining the prior cases of $\pi$, and observing some cancellation, we get
\begin{align*}
f_{1200}&=\frac{1-x}{1-2x}+\frac{x^3}{(1-2x)(1-3x+x^2)}+\frac{x^4(1-x)^2}{(1-2x)^3(1-3x+x^2)}+\frac{x^7(1-x)}{(1-2x)^3(1-3x+x^2)^2}\\
&=\frac{1-x}{1-2x}+\frac{x^3}{(1-3x+x^2)^2},
\end{align*}
as before.

To complete the proof of \eqref{1200the1}, we define a bijection $g$ between $B_n(1200)$ and $B_n(1220)$.  Let $\pi \in B_n(1200)$.  If $\pi$ is binary or does not contain a 0 to the right of its first letter greater than 1, then let $g(\pi)=\pi$.  Otherwise, consider the nondecreasing section $\pi'$ within $\pi$ of the form \eqref{1200form1} or \eqref{1200form2}.  We replace each (maximal) string $w^b$ in $\pi'$, where $w \geq 2$ and $b \geq 1$, with $w0^{b-1}$, and leave all other letters of $\pi$ undisturbed.  Let $g(\pi)$ denote the resulting ascent sequence, which is seen to avoid $1220$.  Combining the two cases above, one may verify that $g$ provides the desired bijection between $B_n(1200)$ and $B_n(1220)$, which completes the proof.
\end{proof}

\section{Further enumerative results}

Outside of the patterns covered in the prior section, the only other $\tau$ of length four such that avoidance of $\tau$ imposes an actual restriction on $B_n$ are given by $$\{0000,0001,0111,1002,1012,1023,1102,1110,1202,1203\}.$$
Note that in each of these patterns, no two positive letters are out of the natural order in going from left to right, and we find $f_\tau$  in this section for each such $\tau$. Taken together with the results from the prior section, one obtains a formula for $f_\tau$ in all cases where $\tau$ is of length four.

We will provide proofs only in the cases above where $\tau$ is binary (i.e., contains only 0's and 1's), and omit the proofs of the others, for the sake of brevity.  (Note that these proofs may be obtained in a comparable manner as those from the prior section, though most require different divisions of the various cases.)  The arguments presented below for the patterns $0111$ and $1110$ will make use of the kernel method.

We first find the generating function $f_{0000}(x)$.

\begin{theorem}\label{0000th}
We have
\begin{equation}\label{0000the1}
f_{0000}(x)=1+\frac{x(1-3x+3x^3+11x^4-3x^5-19x^6-9x^7+11x^8+14x^9+11x^{10}+4x^{11}+x^{12})}{(1-x-x^2-x^3)^5}.
\end{equation}
\end{theorem}
\begin{proof}
Let $\pi \in B(0000)$ and we consider cases based on the number of runs of 0 within $\pi$.  Let $\delta=1-x-x^2-x^3$. If $\pi$ contains one run of zero (or is empty), then the contribution towards $f_{0000}$ is seen to be $\frac{1}{\delta}$, as $\pi$ is a staircase word in this case where each run of a letter has length at most three.  If $\pi$ has two runs of 0, then we have $\pi=00\alpha0\beta$, $0\alpha00\beta$ or $0\alpha0\beta$, where $\alpha \neq \varnothing$ and $\alpha,\beta$ do not contain 0.  If $\beta=\varnothing$, then we get a contribution of $(x^2+2x^3)\cdot\frac{1-\delta}{\delta}=\frac{x^3(1+2x)(1+x+x^2)}{\delta}$, where the $x^2+2x^3$ factor accounts for the 0 entries and $\frac{1-\delta}{\delta}=\frac{x(1+x+x^2)}{1-x-x^2-x^3}$ for the section $\alpha$.
So assume $\beta\neq \varnothing$ and let $a=\max(\alpha)$ and $b=\min(\beta)$.  If $b=a+1$, then, by similar reasoning, we get a contribution of $(x^2+2x^3)\left(\frac{1-\delta}{\delta}\right)^2=\frac{x^4(1+2x)(1+x+x^2)^2}{\delta^2}$.
On the other hand, if $b=a$, then $\beta$ may contain a jump.  Further, if $a$ occurs once in $\alpha$, then it may occur once or twice in $\beta$, whereas if $a$ occurs twice in $\alpha$, it must occur once in $\beta$.  Note $\beta$ is then accounted for by $\frac{x(1+x)}{\delta^2}$ in the former case, and by $\frac{x}{\delta^2}$ in the latter, where the extra factor of $\frac{1}{\delta}$ arises due to $\beta$ possibly containing a jump.  Thus, $\pi$ containing two runs of 0 for which $b=a$ are enumerated by
$$x^2(1+2x)\cdot\frac{x}{\delta}\cdot\frac{x(1+x)}{\delta^2}+x^2(1+2x)\cdot\frac{x^2}{\delta}\cdot\frac{x}{\delta^2}=\frac{x^4(1+2x)^2}{\delta^3}.$$

So assume henceforth $\pi$ contains three runs of 0, and thus we have $\pi=0\alpha0\beta0\gamma$, with $\alpha,\beta \neq \varnothing$ and $\alpha,\beta,\gamma$ not containing 0.  If $\gamma=\varnothing$, then considering cases on whether $b=a$ or $b=a+1$ yields a contribution towards $f_{0000}$ of
$$x^3\cdot\frac{x^2+2x^3}{\delta^2}\cdot\left(1+\frac{1-\delta}{\delta}\right)+x^3\cdot\left(\frac{1-\delta}{\delta}\right)^2=\frac{x^5(1+2x)}{\delta^3}+\frac{x^5(1+x+x^2)^2}{\delta^2},$$
where the $x^2+2x^3$ factor in the first case is attributed to $\alpha$ and $\beta$ sharing the letter $a$.  So assume hereafter $\gamma\neq \varnothing$ and let $c=\max(\beta)$, $d=\min(\gamma)$.  We will consider several cases based on the relative sizes of $a,b,c,d$ as follows. If $b=a+1$ and $d=c+1$, then $\pi$ contains no jumps and we get a contribution of $\left(\frac{x(1-\delta)}{\delta}\right)^3=\frac{x^6(1+x+x^2)^3}{\delta^3}$, as $\alpha,\beta,\gamma$ are mutually disjoint in this case.  If $b=a+1$ and $d=c$, then $\gamma$ may contain a jump, and we get $$x^3\cdot\left(\frac{1-\delta}{\delta}\right)\cdot\frac{x^2+2x^3}{\delta^2}\cdot\frac{1}{\delta}=\frac{x^6(1+2x)(1+x+x^2)}{\delta^4},$$
where the $x^2+2x^3$ accounts for the shared letter between $\beta$ and $\gamma$ and the final $\frac{1}{\delta}$ factor for the possible jump within $\gamma$.

So assume henceforth $b=a$.  If $d=c+1$, then $\beta$ or $\gamma$ may contain a jump, but not both, and considering cases on whether or not $\beta$ contains a jump yields
$$x^3\cdot\frac{(x^2+2x^3)(1-\delta)}{\delta^3}\cdot \frac{1-\delta}{\delta}+x^3\cdot\frac{x^2+2x^3}{\delta^2}\cdot\frac{1-\delta}{\delta^2}=\frac{x^6(1+2x)(1+x+x^2)(1+x+x^2+x^3)}{\delta^4},$$
where the last factor in either case is due to $\gamma$. If $d=c+2$, then $\beta$ has no jump and $\gamma$ is a staircase word, which yields
$$x^3\cdot\frac{x^2+2x^3}{\delta^2}\cdot\frac{1-\delta}{\delta}=\frac{x^6(1+2x)(1+x+x^2)}{\delta^3}.$$

Finally, suppose $d=c$, in which case $\gamma$ can have at most one jump if $\beta$ has a jump and at most two jumps if $\beta$ does not.  Note that if a nondecreasing word $w=w_1\cdots w_m$ contains exactly two jumps, then either there exist two indices $j \in [m-1]$ such that $w_{j+1}=w_j+2$ or one $j$  such that $w_{j+1}=w_j+3$.  We consider further subcases based on the length of $\beta$.  If $|\beta|=1$, then $a=b=d$ and it is seen that there is a contribution of
$$\frac{x^6}{\delta^2}\left(1+2\left(\frac{1-\delta}{\delta}\right)+\left(\frac{1-\delta}{\delta}\right)^2\right)=\frac{x^6}{\delta^4},$$
where the parenthesized factor accounts for the various possibilities concerning the jumps within $\gamma$ of which there are at most two.  So assume $|\beta|>1$, and hence $\beta$ must contain at least two distinct letters as $a=b$ and $c=d$. If $\beta$ has no jump, then we get
$$x^3\cdot\frac{x^2+2x^3}{\delta^2}\cdot\frac{x^2+2x^3}{\delta}\cdot\left(1+\frac{1-\delta}{\delta}\right)^2=\frac{x^7(1+2x)^2}{\delta^5},$$
where the $\frac{x^2+2x^3}{\delta}$ factor accounts for the choice (and positions) of $c=d$ along with the initial section of $\gamma$ prior to any jumps.  On the other hand, if $\beta$ has a jump, then we get
$$x^3\cdot\left(\frac{x^2+2x^3}{\delta^2}\right)^2\cdot\frac{1}{\delta}=\frac{x^7(1+2x)^2}{\delta^5},$$
where one of the $\frac{x^2+2x^3}{\delta^2}$ factors accounts for $\alpha$ and the first part of $\beta$ (prior to the jump) and the other for the remaining section of $\beta$ and the first part of $\gamma$ (prior to its possible jump).

Combining all of the cases above for $\pi$ then gives
\begin{align*}
f_{0000}&=\frac{1}{\delta}+\frac{x^3(1+2x)(1+x+x^2)+x^5(1+x+x^2)^2}{\delta^2}\\
&\quad+\frac{x^4(1+2x)(1+3x)+x^6(1+2x)(1+x+x^2)+x^6(1+x+x^2)^3}{\delta^3}\\
&\quad+\frac{x^6+x^6(1+2x)(1+x+x^2)(2+x+x^2+x^3)}{\delta^4}+\frac{2x^7(1+2x)^2}{\delta^5},
\end{align*}
which may be rewritten as \eqref{0000the1}.
\end{proof}

\begin{theorem}\label{0001th}
We have
\begin{equation}\label{000the1}
f_{0001}(x)=\frac{1-5x+8x^2-10x^4+5x^5+x^6+x^8}{(1-x)^3(1-x-x^2)^3}.
\end{equation}
\end{theorem}
\begin{proof}
The all-zero sequences, together with $\varepsilon$, contribute $\frac{1}{1-x}$ towards the generating function, so assume $\pi \in B(0001)$ contains 1.  Then $\pi$ avoiding $0001$ implies it contains one or two runs of 0 outside of a possible terminal run of 0.  If it contains one such run of 0, then we have $\pi=0\alpha0^r$ or $00\alpha0^r$, where $r \geq0$ and $\alpha$ nonempty does not contain 0. Note that $\pi$ avoiding $021$ and $0001$ implies $\alpha=1^{a_1}\cdots(k-1)^{a_{k-1}}k^{a_k}$ for some $k \geq1$, with $a_i=1$ or $2$ for $1 \leq i \leq k-1$ and $a_k \geq 1$.  Thus, $\pi$ of the stated form are enumerated by
$$\frac{x(1+x)}{1-x}\cdot\frac{1}{1-x-x^2}\cdot\frac{x}{1-x}=\frac{x^2(1+x)}{(1-x)^2(1-x-x^2)},$$
where the factor $\frac{1}{1-x-x^2}$ accounts for the string $1^{a_1}\cdots (k-1)^{a_{k-1}}$ within $\alpha$, the factor $\frac{x(1+x)}{1-x}$ for all of the zeros of $\pi$ and $\frac{x}{1-x}$ for the string $k^{a_k}$.

Otherwise, we have $\pi=0\alpha0\beta0^r$, where $r \geq0$ and $\alpha,\beta$ nonempty do not contain 0.  We consider cases for $\pi$ based on whether $\beta$ contains one or more than one distinct letter as follows:
\begin{align*}
&\bullet\alpha=1^{a_1}\cdots(k-1)^{a_{k-1}}k^s,\,\beta=k^t, \text{ with } s,t \geq 1,\\
&\bullet\alpha=1^{a_1}\cdots(k-1)^{a_{k-1}}k,\,\beta=k\rho,\,\text{ with } \rho\neq\varnothing \text{ a staircase word starting with } k+1 \text{ or } k+2,\\
&\bullet\alpha=1^{a_1}\cdots(k-1)^{a_{k-1}}k,\,\beta=k(k+1)^{a_{k+1}}\cdots \ell^{a_\ell}(\ell+2)^{a_{\ell+2}}\cdots m^{a_m},\,\text{ with } k<\ell<m-1,\\
&\bullet\alpha=1^{a_1}\cdots k^{a_k},\, \beta=(k+1)^{a_{k+1}}\cdots \ell^{a_\ell}, \text{ with } k<\ell,
\end{align*}
where $k \geq 1$ in each case, with $a_i=1$ or 2 if $i$ is not the largest letter of $\pi$ and $a_i\geq1$ unrestricted if $i$ is the largest letter.  These four cases above for $\pi$ are seen to make respective contributions towards $f_{0001}$ of $\frac{x^4}{(1-x)^3(1-x-x^2)}$, $\frac{2x^5}{(1-x)^2(1-x-x^2)^2}$, $\frac{x^6(1+x)}{(1-x)^2(1-x-x^2)^3}$ and $\frac{x^4(1+x)}{(1-x)^2(1-x-x^2)^2}$.
Combining all of the prior cases for $\pi$ implies
\begin{align*}
f_{0001}&=\frac{1}{1-x}+\frac{x^2(1+x)}{(1-x)^2(1-x-x^2)}+\frac{x^4}{(1-x)^3(1-x-x^2)}+ \frac{2x^5}{(1-x)^2(1-x-x^2)^2}\\
&\quad+\frac{x^6(1+x)}{(1-x)^2(1-x-x^2)^3}+\frac{x^4(1+x)}{(1-x)^2(1-x-x^2)^2},
\end{align*}
which yields \eqref{000the1}.
\end{proof}

Let $b_n=b_n(0111)$ for $n \geq1$.  To determine $b_n$, we consider a refinement of the sequence as follows.  Let $b_{n,i}$ for $n \geq 1$ and $0 \leq i \leq n-1$ denote the number of $\pi \in B_n(0111)$ such that $\text{pjum}(\pi)=i$, where pjum is defined as before.  Put $b_{n,i}=0$ if $i \geq n$ or $i<0$.

The $b_{n,i}$ are given recursively as follows.

\begin{lemma}\label{0111lem1}
If $n \geq 3$ and $0 \leq i \leq n-3$, then
\begin{equation}\label{0111lem1e1}
b_{n,i}=b_{n-1,i}+\sum_{j=i}^{n-3}(b_{n-1,j}+b_{n-2,j})+\delta_{i>0}\sum_{d=1}^{n-i-2}\sum_{j=i-1}^{n-d-3}b_{n-d-2,j},
\end{equation}
with $b_{1,0}=1$, $b_{2,0}=2$, $b_{2,1}=0$ and $b_{n,n-1}=b_{n,n-2}=0$ for all $n \geq 3$.
\end{lemma}
\begin{proof}
The initial conditions for $n=1,2$ are apparent, and one may verify the boundary conditions for $i=n-2$ and $i=n-1$.  Let $\mathcal{B}_{n,i}$ denote the subset of $B_n(0111)$ enumerated by $b_{n,i}$.  Suppose $\pi \in \mathcal{B}_{n,i}$ for some $n \geq 3$ and $0 \leq i \leq n-3$.  There are clearly $b_{n-1,i}$ members of $\mathcal{B}_{n,i}$ ending in 0, so assume $\pi$ ends in $k$ for some $k>0$.  Then $\pi$ avoiding 021 implies $k=\max(\pi)$ and $\pi$ avoiding $0111$ implies $k$ appears once or twice.  If $k$ appears as a single run, then $\pi=\pi'k$ or $\pi'kk$, where $\text{pjum}(\pi')=j$ for some $i \leq j \leq n-3$.
Note that $k$ is uniquely determined by $i$ and $j$, together with the largest letter of $\pi'$. Thus, considering all possible $j$ implies that the number of members of $\mathcal{B}_{n,i}$ whose last letter is positive and occurs as a single run is given by $\sum_{j=i}^{n-3}(b_{n-1,j}+b_{n-2,j})$.

Otherwise $\pi=\pi'k0^dk$ for some $d \geq 1$.  Note that we must have $i>0$ in this case due to $k$ occurring twice as an ascent top. Then $\pi' \in \mathcal{B}_{n-d-2,j}$ for some $j \geq i-1$, with $k$ again being determined by $i$, $j$ and $\max(\pi')$.  Considering all possible $d$ and $j$ thus yields the final summation expression in \eqref{0111lem1e1}, where the $\delta_{i>0}$ factor arises from the requirement that $i$ must be positive in order for there to be a nonzero contribution in this case.  Combining the prior cases of $\pi$ above then implies \eqref{0111lem1e1} and completes the proof.
\end{proof}

Define the distribution polynomial $b_n(v)=\sum_{i=0}^{n-1}b_{n,i}v^i$ for $n \geq 1$ and its generating function $B(x,v)=\sum_{n\geq1}b_n(v)x^n$.  Then $B(x,v)$ satisfies the following functional equation.

\begin{lemma}\label{0111lem2}
We have
\begin{equation}\label{0111lem2e1}
\left(1-x-v+2xv+x^2v+\frac{x^3v^2}{1-x}\right)B(x,v)=x(1-v)+x\left(1+x+\frac{x^2v}{1-x}\right)B(x,1).
\end{equation}
\end{lemma}
\begin{proof}
Multiplying both sides of \eqref{0111lem1e1} by $v^i$, and summing over all $0 \leq i \leq n-3$, gives
\begin{align}
b_n(v)&=\sum_{i=0}^{n-3}b_{n-1,i}v^i+\sum_{j=0}^{n-3}(b_{n-1,j}+b_{n-2,j})\sum_{i=0}^jv^i+\sum_{d=1}^{n-3}\sum_{i=1}^{n-3}v^i\sum_{j=i-1}^{n-d-3}b_{n-d-2,j}\notag\\
&=\sum_{i=0}^{n-3}b_{n-1,i}v^i+\sum_{j=0}^{n-3}(b_{n-1,j}+b_{n-2,j})\left(\frac{1-v^{j+1}}{1-v}\right)+\sum_{d=1}^{n-3}\sum_{j=0}^{n-d-3}b_{n-d-2,j}\sum_{i=1}^{j+1}v^i\notag\\
&=b_{n-1}(v)+\frac{1}{1-v}(b_{n-1}(1)+b_{n-2}(1)-vb_{n-1}(v)-vb_{n-2}(v))+\sum_{d=1}^{n-3}\sum_{j=0}^{n-d-3}b_{n-d-2,j}\left(\frac{v-v^{j+2}}{1-v}\right)\notag\\
&=\frac{1-2v}{1-v}b_{n-1}(v)+\frac{1}{1-v}(b_{n-1}(1)+b_{n-2}(1)-vb_{n-2}(v))\notag\\
&\quad+\frac{v}{1-v}\sum_{d=1}^{n-3}(b_{n-d-2}(1)-vb_{n-d-2}(v)), \qquad n \geq 3. \label{0111lem1e2}
\end{align}
Note that \eqref{0111lem1e2} is also seen to hold for $n=2$, upon taking $b_0(v)=0$.  Multiplying both sides of \eqref{0111lem1e2} by $x^n$, and summing over all $n \geq 2$, gives
\begin{align*}
B(x,v)-x&=\frac{x(1-2v)}{1-v}B(x,v)+\frac{x(1+x)}{1-v}B(x,1)-\frac{x^2v}{1-v}B(x,v)\\
&\quad+\frac{v}{1-v}\sum_{d\geq1}\sum_{n\geq d+3}(b_{n-d-2}(1)-vb_{n-d-2}(v))x^n\\
&=\frac{x(1+x)}{1-v}B(x,1)+\frac{x(1-2v-xv)}{1-v}B(x,v)+\frac{v}{1-v}\sum_{d\geq1}x^{d+2}(B(x,1)-vB(x,v))\\
&=\frac{x\left(1+x+\frac{x^2v}{1-x}\right)}{1-v}B(x,1)+\frac{x\left(1-2v-xv-\frac{x^2v^2}{1-x}\right)}{1-v}B(x,v),
\end{align*}
which may be rewritten as \eqref{0111lem2e1}.
\end{proof}

\begin{theorem}\label{0111th}
We have
\begin{equation}\label{0111the1}
f_{0111}(x)=\frac{(1-x)^2-\sqrt{1-4x+2x^2+x^4}}{2x^2}.
\end{equation}
\end{theorem}
\begin{proof}
We apply the kernel method and let $v=v_0$ in \eqref{0111lem2e1}, where $v_0$ satisfies
$$x^3v_0^2-(1-x)(1-2x-x^2)v_0+(1-x)^2=0.$$
This cancels out the left-hand side of \eqref{0111lem2e1} and solving for $B(x,1)$ implies
$$B(x,1)=\frac{(1-x)(v_0-1)}{1-x^2+x^2v_0}=\frac{x}{1-x}v_0,$$
with $v_0$ given by
$$v_0=\frac{(1-x)\left(1-2x-x^2-\sqrt{1-4x+2x^2+x^4}\right)}{2x^3}.$$
Note that we have chosen the negative root for $v_0$ since we seek a power series solution to $B(x,1)$.
Formula \eqref{0111the1} now follows from the fact $f_{0111}=1+B(x,1)$.
\end{proof}

\noindent \emph{Remark:}  It is seen, upon comparing generating functions, that $b_{n-1}=A082582(n)$ for all $n \geq1$.  Hence, we have an apparently new combinatorial interpretation of A082582, which is known already to enumerate such discrete structures as Dyck paths of semi-length $n$ that avoid $UUDD$ or bargraphs with semi-perimeter $n$.

Our next result involving the pattern $1110$ will draw upon the preceding one in a fundamental way.

\begin{theorem}\label{1110th}
We have
\begin{equation}\label{1110the1}
f_{1110}(x)=\frac{1-5x+8x^2-5x^3+x^4-(1-3x+x^2+2x^3)\sqrt{1-4x+2x^2+x^4}}{2x^2(1-2x)^2}.
\end{equation}
\end{theorem}
\begin{proof}
Suppose $\pi \in B(1110)-B(0111)$ and we consider the smallest $k>0$ which appears at least thrice.  Then $\pi$ has one of the following forms:
\begin{align*}
\text{(i)}&\, \pi=\pi'k^r\alpha, \quad r \geq 3,\\
\text{(ii)}&\,\pi=\pi'k0^ak^2\beta \text{ or } \pi=\pi'k^20^ak\beta, \quad a\geq1,\\
\text{(iii)}&\,\pi=\pi'k0^bk0^ck\gamma, \quad b,c \geq1,
\end{align*}
where $\pi'$ is a nonempty member of $B(0111)$, $k>\max(\pi')$ and $\alpha,\beta,\gamma$ are possibly empty.  Note $\pi$ avoiding $1110$  implies $\alpha,\beta,\gamma$ cannot contain $0$, and hence are nondecreasing.
Given $0 \leq i \leq j$, note that if $\text{pjum}(\pi')=j$, then $k>\max(\pi')$ may be chosen so that there are at most $j-i$ jumps within the nondecreasing section $k^{r-2}\alpha$ in $\pi$ of the form (i), and hence this section is accounted for by $h_{j-i}(x)$.  Similarly, there are at most $j-i+1$ and $j-i+2$ jumps in the terminal sections $k\beta$ and $k\gamma$ within $\pi$ of the forms (ii) or (iii).

Let $f_j(x)=[y^j]B(x,y)$ for $j\geq0$, where $B(x,y)$ is as defined above. Considering all possible $i$ and $j$, it is seen that $\pi$ of the form (i) are enumerated by
\begin{align*}
\sum_{j\geq0}\sum_{i=0}^jf_j(x)\cdot x^2h_{j-i}(x)&=\frac{x^3}{1-2x}\sum_{j\geq0}\sum_{i=0}^jf_j(x)\left(\frac{1-x}{1-2x}\right)^{j-i}=\frac{x^3}{1-2x}\sum_{j\geq0}f_j(x)\left(\frac{1-\left(\frac{1-x}{1-2x}\right)^{j+1}}{1-\frac{1-x}{1-2x}}\right)\\
&=x^2\sum_{j\geq0}f_j(x)\left(\left(\frac{1-x}{1-2x}\right)^{j+1}-1\right)=x^2\left(\frac{1-x}{1-2x}B\left(x,\frac{1-x}{1-2x}\right)-B(x,1)\right).
\end{align*}
Similarly, the contributions from $\pi$ of the forms (ii) or (iii) towards $f_{1110}$ are given respectively by
\begin{align*}
\sum_{j\geq0}\sum_{i=0}^jf_j(x)\cdot \frac{2x^3}{1-x}h_{j-i+1}(x)&=\frac{2x^4}{(1-2x)^2}\sum_{j\geq0}\sum_{i=0}^jf_j(x)\left(\frac{1-x}{1-2x}\right)^{j-i}\\
&=\frac{2x^3}{1-2x}\left(\frac{1-x}{1-2x}B\left(x,\frac{1-x}{1-2x}\right)-B(x,1)\right)
\end{align*}
and
\begin{align*}
\sum_{j\geq0}\sum_{i=0}^jf_j(x)\cdot \frac{x^4}{(1-x)^2}h_{j-i+2}(x)&=\frac{x^4}{(1-2x)^2}\left(\frac{1-x}{1-2x}B\left(x,\frac{1-x}{1-2x}\right)-B(x,1)\right).
\end{align*}
Combining the contributions from $\pi$ of the forms (i)--(iii) above implies members of $B(1110)-B(0111)$ are enumerated by
\begin{align*}
&x^2\left(1+\frac{2x}{1-2x}+\frac{x^2}{(1-2x)^2}\right)\left(\frac{1-x}{1-2x}B\left(x,\frac{1-x}{1-2x}\right)-B(x,1)\right)\\
&=\frac{x^2(1-x)^3}{(1-2x)^3}B\left(x,\frac{1-x}{1-2x}\right)-\frac{x^2(1-x)^2}{(1-2x)^2}B(x,1).
\end{align*}

Adding the contribution from members of $B(0111)$, namely, $1+B(x,1)$, we get
\begin{equation}\label{1110the2}
f_{1110}=1+\frac{1-4x+3x^2+2x^3-x^4}{(1-2x)^2}B(x,1)+\frac{x^2(1-x)^3}{(1-2x)^3}B\left(x,\frac{1-x}{1-2x}\right).
\end{equation}
Solving for $B(x,v)$ in \eqref{0111lem2e1}, and making use of the formula found above for $B(x,1)$, we have
$$B(x,v)=\frac{2x^2(1-x)(1-v)+(1-x^2+x^2v)(1-2x-x^2)-(1-x^2+x^2v)\sqrt{1-4x+2x^2+x^4}}{2x((1-x)^2-(1-x)(1-2x-x^2)v+x^3v^2)},$$
and hence
$$B\left(x,\frac{1-x}{1-2x}\right)=\frac{(1-2x)\left(1-3x+x^3+x^4-(1-x-x^2)\sqrt{1-4x+2x^2+x^4}\right)}{2x^3(1-x)^2}.$$
Substituting this formula into \eqref{1110the2}, and simplifying, leads to \eqref{1110the1} after several algebraic steps.
\end{proof}

We conclude by stating without proof the formulas for $f_\tau$ in the remaining cases of $\tau$.

\begin{theorem}\label{other}
We have
\begin{align*}
f_{1002}(x)&=\frac{1-9x+34x^2-69x^3+81x^4-54x^5+16x^6+x^7}{(1-x)^4(1-2x)^3},\\
f_{1012}(x)&=\frac{1-5x+8x^2-4x^3+x^4}{(1-x)(1-2x)(1-3x+x^2)},\\
f_{1023}(x)&=\frac{1-8x+26x^2-43x^3+38x^4-16x^5+x^7}{(1-x)^3(1-2x)^3},\\
f_{1102}(x)&=\frac{1-8x+26x^2-43x^3+38x^4-16x^5+x^6}{(1-x)^3(1-2x)^3},\\
f_{1202}(x)&=\frac{1-3x+x^2}{(1-x)(1-3x)},\\
f_{1203}(x)&=\frac{1-8x+24x^2-32x^3+17x^4-2x^5-x^6}{(1-2x)^3(1-3x+x^2)}.
\end{align*}
\end{theorem}

\noindent\emph{Remarks:} From the formula for $f_{1202}$, we have $b_n(1202)=\frac{3^{n-1}+1}{2}$ for $n\geq1$, which coincides with A007051 in \cite{Sloane}, offset by one.  The relatively simple formula for $f_{1202}$, compared with the other $f_\tau$, arises from some cancellation which occurs in the later steps of the proof that we found for it.  A simpler algebraic argument may be possible, or better yet, perhaps a direct bijection can be found with one of the structures already known to be enumerated by A007051.  Also, notice the uncanny similarity in the formulas for $f_{1023}$ and $f_{1102}$, with their only being a slight difference in one of the terms in the numerator.  Perhaps this could be explained via a ``near bijection'' between $B_n(1023)$ and $B_n(1102)$ which would isolate the source of the difference.



\begin{thebibliography}{99}

\bibitem{BP}
A. M. Baxter and L. K. Pudwell, Ascent sequences avoiding pairs of patterns, \emph{Electron. J. Combin.}
{\bf22}(1) (2015), \#P1.58.

\bibitem{BCD}
M. Bousquet-M{\'e}lou, A. Claesson, M. Dukes and S. Kitaev, (2+2)-free posets, ascent sequences and pattern avoiding permutations, \emph{J. Combin. Theory Ser. A} {\bf117}(7) (2010), 884--909.

\bibitem{CM0}
D. Callan and T. Mansour, Ascent sequences avoiding a triple of 3-letter patterns and Fibonacci numbers, \emph{Art Discrete Appl. Math.} {\bf7} (2024), \#P2.10.

\bibitem{CM}
D. Callan and T. Mansour, Ascent sequences avoiding a triple of length-3 patterns, \emph{Electron. J. Combin.} {\bf32}(1) (2025), \#P1.40.

\bibitem{CMS}
D. Callan, T. Mansour and M. Shattuck, Restricted ascent sequences and Catalan numbers, \emph{Appl. Anal. Discrete Math.} {\bf 8} (2014), 288--303.

\bibitem{Chen}
W. Y. C. Chen, A. Y. L. Dai, T. Dokos, T. Dwyer and B. E. Sagan, On $021$-avoiding ascent sequences, \emph{Electron. J. Combin.} {\bf20}(1) (2013), \#P76.

\bibitem{CDK}
A. Claesson, M. Dukes and S. Kitaev, A direct encoding of Stoimenow's matchings
as ascent sequences, \emph{Australas. J. Combin.} {\bf 49} (2011), 47--59.

\bibitem{DP}
M. Dukes and R. Parviainen, Ascent sequences and upper triangular matrices containing non-negative integers, \emph{Electron. J. Combin.} {\bf17}(1) (2010), \#R53.

\bibitem{DS}
P. Duncan and E. Steingr\'{\i}msson, Pattern avoidance in ascent sequences, {\em Electron. J. Combin.} {\bf18}(1) (2011), \#P226.

\bibitem{HM}
Q. Hou and T. Mansour, Kernel method and linear recurrence system, \emph{J. Comput. Appl. Math.} {\bf 261}:1 (2008), 227--242.

\bibitem{JMS}
V. Jel\'\i nek, T. Mansour and M. Shattuck, On multiple pattern avoiding set partitions, \emph{Adv. in Appl. Math.} {\bf 50} (2013), 292--326.

\bibitem{K}
S. Kitaev, {\it Patterns in Permutations and Words}, Springer-Verlag, 2011.

\bibitem{KR}
S. Kitaev and J. Remmel, Enumerating (2+2)-free posets by the number of minimal elements and other statistics, \emph{Discrete Appl. Math.} {\bf159} (2011), 2098--2108.

\bibitem{MS1}
T. Mansour and M. Shattuck, Some enumerative results related to ascent sequences, \emph{Discrete Math.} {\bf 315-316} (2014), 29--41.

\bibitem{MS2}
T. Mansour and M. Shattuck, Ascent sequences and Fibonacci numbers, \emph{Filomat} {\bf29}:4 (2015), 703--712.

\bibitem{MY}
T. Mansour and G. Y{\i}ld{\i}r{\i}m, Inversion sequences avoiding 021 and another pattern of length four, \emph{Discrete Math. Theor. Comput. Sci.} {\bf25}:2 (2023), Art \#20.

\bibitem{Sloane}
N. J. A. Sloane et al., The On-Line Encyclopedia of Integer Sequences, 2025. Available at https://oeis.org.

\bibitem{Y}
S. H. F. Yan, Ascent sequences and 3-nonnesting set partitions, \emph{European J. Combin.} {\bf39} (2014), 80--94.


\end{thebibliography}
\end{document}